\documentclass[a4paper,11pt]{amsart}
\usepackage[T1]{fontenc}
\usepackage[english]{babel}
\usepackage[cp1252]{inputenc}
\usepackage{amsthm}
\usepackage{amsmath}
\usepackage{amsfonts}

\usepackage{tikz-cd}
  \usepackage{pgfplots}
\usepackage{microtype}

\usepackage[margin=0.8in]{geometry}

\usepackage{xcolor}
 
\pgfplotsset{compat=1.18}
\usepgfplotslibrary{fillbetween}

\usepackage{amssymb}
\usepackage{hyperref}
\usepackage{color}
\usepackage{caption}
\usepackage{indentfirst}
\usepackage{amssymb}
\usepackage{eufrak}
\usepackage{mathrsfs}
\usepackage{xypic}

\usepackage{booktabs,array,graphicx,longtable}
\usepackage{array,booktabs,tabularx}

\usepackage{booktabs}
\usepackage{listings}

\theoremstyle{plain}   

\usepackage{booktabs}
\usepackage{graphicx}
\usepackage{array}

\newcommand{\R}{\mathbb R}
\newcommand{\C}{\mathbb C}

\newcommand{\Z}{\mathbb Z}

\newcommand{\Log}{\mathrm{Log}\,}

\newcommand{\VolZ}{\operatorname{Vol}_{\mathbb Z}}

\newcommand{\length}{\operatorname{length}}

\newcommand{\di}{\displaystyle}

\newcommand{\T}{(\mathbb C^*)}

\newcommand{\Crit}{\operatorname{Crit}}

\newcommand{\Area}{\operatorname{Area}}

\newcommand{\supp}{\operatorname{supp}}

\newcommand{\Newt}{\operatorname{Newt}}

\newcommand{\card}{\#}
\newcommand{\calA}{\mathscr A}
\newcommand{\calC}{\mathcal C}
\newcommand{\calS}{\mathcal S}

\usepackage{array}
\usepackage{tabularx}

\newcommand{\Sing}{\operatorname{Sing}}

\newcommand\restr[2]{{
  \left.\kern-\nulldelimiterspace 
  #1 
  \right|_{#2} 
  }}

\newtheorem{remark}{Remark}[section]
\newtheorem*{mtheorem*}{Main Theorem}
\newtheorem{theorem}{Theorem}[section]
\newtheorem{definition}{Definition}[section]
\newtheorem{proposition}{Proposition}[section]

\newtheorem{lemma}{Lemma}[section]

\begin{document}

\title{Cusps and nodes  of Amoeba Contours}

\author{Mounir Nisse}

\address{Mounir Nisse\\
Department of Mathematics, Xiamen University Malaysia, Jalan Sunsuria, Bandar Sunsuria, 43900, Sepang, Selangor, Malaysia.
}
\email{mounir.nisse@gmail.com, mounir.nisse@xmu.edu.my}

\date{}

\thanks{This research is supported in part by Xiamen University Malaysia Research Fund (Grant no. XMUMRF/ 2024-C5/IMAT/0013).}

 
\maketitle

\begin{abstract}
We establish new Newton-polygon bounds for the singularities of amoeba contours of smooth plane curves. Our cusp estimate refines the degree-four bound of Lang--Shapiro--Shustin, reducing its leading coefficient from $8$ to $4$ while retaining the normalized area, boundary lattice points, and directional widths of the Newton polygon. We also obtain a new multiplicity-sensitive bound for transverse $s$-nodes, with the natural decay factor $1/\binom{s}{2}$. The proofs combine logarithmic Gauss maps, normalized fiber products, ramification theory, and saturated off-diagonal intersection schemes, revealing the distinct geometric mechanisms governing cusps and multiple nodes.
\end{abstract}

\section*{Introduction}

Amoebas translate complex algebraic geometry into a real geometric setting while retaining remarkably precise information about Newton polytopes, toric compactifications, and logarithmic asymptotics.  If $f\in\mathbb C[z^{\pm1},w^{\pm1}]$ is a Laurent polynomial and $C_f=\{f=0\}\subset(\mathbb C^*)^2$, its amoeba is the image $\mathcal A_f=\Log (C_f)$ under the logarithmic map $\Log(z,w)=(\log|z|,\log|w|)$.  The contour $\mathcal K_f$ is the set of critical values of $\Log|_{C_f}$.  It contains the boundary of the amoeba, but in general it is substantially richer: it records the folds of the logarithmic projection, the interaction of distinct critical branches, and the singular values at which the local geometry of the projection changes.  Through the logarithmic Gauss map, the contour is simultaneously linked to real algebraic geometry, ramification theory, and the combinatorics of the Newton polygon.  These links make its singularities natural invariants of a plane curve in the algebraic torus.

The quantitative geometry of amoeba contours was advanced decisively by Lang, Shapiro, and Shustin \cite{LangShapiroShustin21}.  They introduced the $\mathbb R$-degree of the contour and obtained general bounds for its intersections with affine lines.  In the plane-curve case, their argument also produced an explicit estimate for the set $\Sigma_H$ of critical points at which the contour parametrization is singular.  Under their smoothness hypothesis, Corollary~2.11 of \cite{LangShapiroShustin21} states that
$$
\#\Sigma_H\leq 2d^3(4d-2)-\VolZ(\Delta),
$$
where $d$ is the projective degree and $\VolZ(\Delta)=2\Area(\Delta)$ is the normalized area of the Newton polygon.  This is an order-$d^4$ estimate with leading coefficient $8$.  Its proof places the logarithmic criticality and ramification conditions in a real polynomial system and applies a total-degree B\'ezout count.  It was the first general estimate of this type and made clear that cusp counting is a central part of the quantitative theory of amoeba contours.

The present work develops a different approach to this problem.  Instead of treating the ramification equations only through their ambient total degrees, we use the intrinsic geometry of the logarithmic Gauss map and the normalized complexified critical curve.  This retains the numerical invariants of the Newton polygon throughout the calculation.  Let $\Delta$ be a two-dimensional lattice polygon, let $n=\VolZ(\Delta)$, let $b=|\partial\Delta\cap\mathbb Z^2|$, and let $\omega_z$ and $\omega_w$ be its two coordinate lattice widths.  We consider the class $\mathscr C_\Delta$ of smooth curves $C_f$ whose logarithmic critical locus is smooth and whose compactification in the toric surface $X_\Delta$ is nondegenerate.  When the ramification points of the contour map are isolated, we denote by $\kappa(f)$ the number of distinct cusp values of $\mathcal K_f$.

The first principal result of this work is the Newton-polygon estimate $\kappa(f)\leq 2n(2n-b)+4n(\omega_z+\omega_w)$.  This formula is sensitive not only to the area of $\Delta$, but also to its boundary and its directional widths.  It therefore distinguishes polygons having the same normalized area but different shapes, a distinction that disappears from a total-degree calculation.  If neither logarithmic coordinate function is constant on any irreducible component of the normalized complexified critical curve, the estimate improves to $\kappa(f)\leq 2n(2n-b)+4n\min\{\omega_z,\omega_w\}$.  The bound applies to isolated higher-order ramification values as well, since each such value consumes at least one zero of the corresponding logarithmic differential.

For a polynomial of projective degree $d$, the containment $\Delta\subset d\Delta_2$ converts the polygonal estimate into the universal inequality $\kappa(f)\leq4d^4+8d^3-6d^2$.  In particular, $\kappa(f)\leq(13/2)d^4$ for every $d$, $\kappa(f)\leq6d^4$ for $d\geq3$, and $\kappa(f)\leq(4+O(d^{-1}))d^4$ asymptotically.  When $\Delta=d\Delta_2$, the exact substitution $n=d^2$, $b=3d$, and $\omega_z=\omega_w=d$ gives the sharper formula $\kappa(f)\leq4d^4+2d^3$.  Under the additional nonconstancy hypothesis, the same full-triangle case yields $\kappa(f)\leq4d^4-2d^3$.

This gives a direct and substantial comparison with Corollary~2.11 of Lang--Shapiro--Shustin.  Their estimate has leading term $8d^4$, whereas the present estimate has leading term $4d^4$.  Thus the leading coefficient is reduced by one half.  For the full degree triangle, their bound reads $8d^4-4d^3-d^2$, while our polygonal calculation gives $4d^4+2d^3$, or $4d^4-2d^3$ under the componentwise nonconstancy condition.  The improvement does not arise from a different estimate for the same total-degree system.  It comes from replacing that ambient calculation by a divisor count on the normalization of the logarithmic-Gauss fiber product.  Riemann--Hurwitz controls the genus contribution, while the pole divisors of the meromorphic functions $Z=z\zeta$ and $W=w\eta$ are governed by the coordinate widths of $\Delta$.  Cusps are then detected as common zeros of the logarithmic differentials $d\log Z$ and $d\log W$.  This geometric reformulation explains both the smaller leading coefficient and the appearance of the finer Newton-polygon data.

The cusp estimate addresses ramification of a single critical branch.  A second and independent source of contour singularities occurs when distinct regular critical branches have the same logarithmic image.  This work introduces a systematic bound for such coincidences and, more generally, for transverse multiple points.  An $s$-node is a contour value having exactly $s$ distinct critical lifts at which the contour map is immersive and whose $s$ tangent lines are pairwise distinct.  If $N_s(f)$ denotes the number of isolated $s$-nodes, our second principal result is
$$
N_s(f)\leq\left\lfloor\frac{8n^2\omega_z\omega_w}{s(s-1)}\right\rfloor.
$$
Since $\omega_z,\omega_w\leq n$, the area-only consequence is $N_s(f)\leq\lfloor8n^4/(s(s-1))\rfloor$.  Equivalently, it is $N_s(f)\leq\lfloor128\Area(\Delta)^4/(s(s-1))\rfloor$.  If $f$ has projective degree $d$, then $n\leq d^2$ and $\omega_z,\omega_w\leq d$.  This gives $N_s(f)\leq\lfloor8d^6/(s(s-1))\rfloor$.

The factor $1/\binom{s}{2}$ is an essential feature of the result rather than a formal adjustment.  A transverse intersection of $s$ branches determines exactly $\binom{s}{2}$ unordered pairs of distinct critical lifts with the same logarithmic value.  The normalized complexified critical curve carries the meromorphic map $(Z,W)$, and isolated off-diagonal coincidences are detected on its self-product by the two equations $Z(q_1)=Z(q_2)$ and $W(q_1)=W(q_2)$.  Their residual intersection, after removing the diagonal and all positive-dimensional overlap components, is bounded by the product of the degrees of $Z$ and $W$.  The estimates $\deg Z\leq2n\omega_z$ and $\deg W\leq2n\omega_w$ then give at most $4n^2\omega_z\omega_w$ unordered off-diagonal pairs.  Dividing this budget by $\binom{s}{2}$ produces the stated bound.  In particular, the theorem becomes stronger as the multiplicity $s$ increases and records a geometric constraint that cannot be read from a bound for the total number of singular values alone.

This separation between ramification singularities and coincidence singularities is one of the main structural contributions of the paper.  Cusps are governed by zeros of logarithmic differentials on a normalized fiber product, whereas $s$-nodes are governed by a saturated two-lift self-intersection problem.  The diagonal saturation is indispensable: without it, every point paired with itself forms an automatic positive-dimensional component and obscures the off-diagonal solutions.  Saturation by the diagonal, followed when necessary by the removal of toric-boundary and positive-dimensional coincidence components, isolates precisely the residual scheme relevant to distinct critical lifts.  The resulting method is compatible with exact elimination and real-root isolation, so the theoretical bounds also provide a framework for certified computations in explicit examples.

The two main estimates therefore refine the quantitative study initiated in \cite{LangShapiroShustin21} in complementary directions.  The cusp theorem improves the leading degree coefficient, retains the geometry of the Newton polygon, and identifies the divisor responsible for ramification.  The $s$-node theorem supplies a new multiplicity-sensitive estimate for transverse intersections of several contour branches.  Together they show that the singular geometry of an amoeba contour is controlled not by one undifferentiated algebraic system, but by two distinct mechanisms with different intersection theories.  This point of view connects the local classification of contour singularities with global toric invariants and opens a route toward sharper bounds for restricted polygonal families, exact counts for particular curves, and asymptotic improvements obtained from the real part of the logarithmic critical curve.

\medskip


Let $\Delta\subset\R^2$ be a fixed two-dimensional lattice polygon and
put $A_\Delta=\Delta\cap\Z^2=\{\alpha_1,\ldots,\alpha_s\}$, where
$\alpha_j=(a_j,b_j)$.  The clarification in the question means that the
coefficient space is
$$
\calS_\Delta=
\left\{f_c(z,w)=\sum_{\alpha=(a,b)\in A_\Delta}
c_\alpha z^aw^b:
c_\alpha\neq0\ \text{for every }\alpha\in\operatorname{Vert}(\Delta)
\right\}.
$$
Equivalently,
$\operatorname{Vert}(\Delta)\subseteq\supp(f)\subseteq
\Delta\cap\Z^2$.  If $r=\#\operatorname{Vert}(\Delta)$ and
$s=\#(\Delta\cap\Z^2)$, then
$\calS_\Delta\simeq(\C^\ast)^r\times\C^{s-r}$.  The coefficients at the
vertices are mandatory and nonzero, whereas the coefficients at
nonvertex lattice points are optional.  In particular, every
$f\in\calS_\Delta$ has actual Newton polygon
$\Delta_f=\operatorname{conv}(\supp(f))=\Delta$.
For $V_f=\{f=0\}\subset\T$, set $\Log(z,w)=(\log|z|,\log|w|)$,
$\calA_f=\Log(V_f)$, and, when $V_f$ is smooth,
$\calC(\calA_f)=\Log(\Crit(\Log|_{V_f}))$.

%
%

Set $f_z=\partial f/\partial z$ and $f_w=\partial f/\partial w$.  On the
smooth curve $V_f$, the logarithmic Gauss map is
$\gamma_f:V_f\longrightarrow\mathbb P^1_\C$,
$\gamma_f(z,w)=[zf_z(z,w):wf_w(z,w)]$.  A standard calculation gives
$\Sigma_f=\gamma_f^{-1}(\mathbb P^1_\R)$.  Indeed, a tangent vector
$(\dot z,\dot w)$ to $V_f$ satisfies
$f_z\dot z+f_w\dot w=0$, while the differential of $\Log$ sends it to
$(\operatorname{Re}(\dot z/z),\operatorname{Re}(\dot w/w))$.  The real
rank is smaller than two exactly when $zf_z/(wf_w)$ is real, with the
usual interpretation at $wf_w=0$.


All numerical calculations were performed in Python, principally with
NumPy, SciPy, SymPy, and Matplotlib. The numerical data thus obtained were
subsequently used to prepare the figures in PGFPlots and \LaTeX{}.

\medskip

\noindent{\it Acknowledgements.}
The author gratefully thanks Boris Shapiro for sharing the paper written jointly
with Lionel Lang and Eugenii Shustin. This paper, together with the questions
formulated in it, provided the principal motivation for the present work.
 
\section{An Upper Bound for the Isolated Singularities of an Amoeba Contour}

Define
$x_{\min}=\min\{a:(a,b)\in\Delta\}$ and
$y_{\min}=\min\{b:(a,b)\in\Delta\}$, and put
$d_\Delta=\max\{(a-x_{\min})+(b-y_{\min}):(a,b)\in\Delta\}$.  Multiplication
of $f$ by the invertible Laurent monomial $z^{-x_{\min}}w^{-y_{\min}}$ does
not change $V_f\subset\T$, its logarithmic critical locus, or its contour.
After that multiplication, $f$ is an ordinary polynomial of total degree at
most $d_\Delta$.  The number $d_\Delta$ therefore gives a uniform degree
enclosure depending only on $\Delta$.

\subsection*{Complexification of the logarithmic critical locus}

Write $f(z,w)=\sum c_{ab}z^aw^b$ after the preceding monomial translation
and introduce independent variables $\zeta,\eta$.  Define
$f^\dagger(\zeta,\eta)=\sum\overline{c_{ab}}\zeta^a\eta^b$ and set
$P=zf_z$, $Q=wf_w$, $P^\dagger=\zeta f^\dagger_\zeta$, and
$Q^\dagger=\eta f^\dagger_\eta$.  The polynomial
$K=PQ^\dagger-QP^\dagger$ has total degree at most $2d_\Delta$ in
$(z,w,\zeta,\eta)$.

The complexified logarithmic critical curve is
$\Gamma_f=V(f,f^\dagger,K)\subset(\C^*)^4$.  On the antiholomorphic real
form $\zeta=\bar z$ and $\eta=\bar w$, the equation $K=0$ says that
$[zf_z:wf_w]\in\mathbb P^1_\R$.  This is precisely the criticality
condition for $\Log|_{V_f}$ when $V_f$ is smooth.  Indeed, if
$(\dot z,\dot w)$ is tangent to $V_f$, then
$f_z\dot z+f_w\dot w=0$, whereas
$d\Log(\dot z,\dot w)=(\operatorname{Re}(\dot z/z),
\operatorname{Re}(\dot w/w))$.  The real rank drops exactly when the
homogeneous ratio $[zf_z:wf_w]$ is real.

On $\Gamma_f$ consider the algebraic critical-value map
$\mu(z,w,\zeta,\eta)=(R,S)=(z\zeta,w\eta)$.  On the real form one has
$R=|z|^2$ and $S=|w|^2$, and the logarithmic contour map is obtained by
composing $\mu$ with $(R,S)\mapsto(\frac12\log R,\frac12\log S)$.  The
latter is a real-analytic diffeomorphism of $\R_{>0}^2$ onto $\R^2$ and
therefore does not change local singularity types.

\subsection*{The two algebraic sources of isolated singular values}

At a smooth complete-intersection point $p$ of $\Gamma_f$, i.e. 
$p$ satisfies
\(
 f(p)=f^\dagger(p)=K(p)=0,
 \, 
 \operatorname{rank}\mathcal J_{\Gamma_f}(p)=3.
\)
Thus, 
 its tangent line is
the common kernel of $df$, $df^\dagger$, and $dK$ (see Appendix A).  Define
$$
J_R=\det\frac{\partial(f,f^\dagger,K,R)}
{\partial(z,w,\zeta,\eta)},\qquad
J_S=\det\frac{\partial(f,f^\dagger,K,S)}
{\partial(z,w,\zeta,\eta)}.
$$
The restriction of $d\mu$ to the tangent line vanishes exactly when
$J_R=J_S=0$.  After removal of components introduced by a singular
Jacobian presentation and components on the toric boundary, these equations
define the ramification scheme $Z_{\mathrm{ram}}$.

For double values, take two copies of $\Gamma_f$ with coordinates
$p_i=(z_i,w_i,\zeta_i,\eta_i)$, $i=1,2$, and impose
$R_1-R_2=z_1\zeta_1-z_2\zeta_2=0$ and
$S_1-S_2=w_1\eta_1-w_2\eta_2=0$.  The resulting eight equations are
$$
f_1=f_1^\dagger=K_1=f_2=f_2^\dagger=K_2=R_1-R_2=S_1-S_2=0.
$$
The diagonal ideal is
$I_{\mathrm{diag}}=(z_1-z_2,w_1-w_2,\zeta_1-\zeta_2,
\eta_1-\eta_2)$.  Saturating by $I_{\mathrm{diag}}$ removes the automatic
solutions $p_1=p_2$, and saturation by the irrelevant ideals of a toric
compactification removes boundary components.  The remaining scheme is the
ordered off-diagonal double-point scheme $Z_{\mathrm{off}}$.

\begin{lemma}
Assume that $Z_{\mathrm{ram}}$ and $Z_{\mathrm{off}}$ are
zero-dimensional and that every isolated singular value of
$\calC(\calA_f)$ is produced either by ramification of $\mu$ or by two
distinct critical preimages.  Then
$$
\card\Sing_{\mathrm{isol}}\calC(\calA_f)
\leq\length(Z_{\mathrm{ram}})+\frac12\length(Z_{\mathrm{off}}).
$$
\end{lemma}

\begin{proof}
Every singular value having a ramified lift has at least one geometric point
of $Z_{\mathrm{ram}}$ above it.  Distinct values require distinct source
points, and the number of geometric points of a zero-dimensional scheme is
at most its scheme-theoretic length.  Hence the number of such values is at
most $\length(Z_{\mathrm{ram}})$.
Every remaining isolated singular value has two distinct lifts $p$ and $q$.
It gives the two distinct ordered points $(p,q)$ and $(q,p)$ of
$Z_{\mathrm{off}}$.  Ordered pairs lying over different values are distinct.
Consequently, the number of these remaining singular values is at most
$\frac12\length(Z_{\mathrm{off}})$.  Adding the two estimates proves the
claim.
\end{proof}

\subsection*{Total-degree estimates}

Put $d=d_\Delta$.  The six critical-curve equations occurring in the
two-lift system have degrees $d,d,2d,d,d,2d$, and the two equal-value
equations have degree two.  If the saturated off-diagonal scheme is
zero-dimensional, affine B\'ezout's theorem, or equivalently the
projective B\'ezout estimate after homogenization, gives
$$
\length(Z_{\mathrm{off}})
\leq d\cdot d\cdot2d\cdot d\cdot d\cdot2d\cdot2\cdot2
=16d^6.
$$
Saturation can only discard components or local contributions; it cannot
increase the length of the isolated torus solutions.  Therefore the
off-diagonal contribution to the number of image values is at most $8d^6$.
For the ramification contribution, note that the rows in the determinant
$J_R$ have degrees at most $d-1,d-1,2d-1,1$.  Every determinant term is a
product of one entry from each row, so $\deg J_R\leq4d-2$.  The same estimate
holds for $J_S$.  Choose a generic pair $(\lambda,\nu)\in\C^2$ and put
$H=\lambda J_R+\nu J_S$.  On each irreducible component of the reduced curve
$\Gamma_f$, the assumption that the ramification scheme is
zero-dimensional guarantees that $J_R$ and $J_S$ do not both vanish
identically.  A generic choice of $(\lambda,\nu)$ therefore makes $H$
nonzero on every component.  The intersection
$V(f,f^\dagger,K,H)$ is then zero-dimensional along the relevant torus
curve, contains $Z_{\mathrm{ram}}$, and has B\'ezout length at most
$
d\cdot d\cdot2d\cdot(4d-2)=2d^3(4d-2)=8d^4-4d^3.
$
It follows that
$\length(Z_{\mathrm{ram}})\leq8d^4-4d^3$.
 
\begin{theorem}
Let $\Delta$ be a two-dimensional lattice polygon and let
$f\in\calS_\Delta$.  Assume that $V_f$ is smooth in $\T$, that
$\Gamma_f$ is reduced, and that the ramification scheme and the saturated
off-diagonal double-point scheme are zero-dimensional after removal of the
diagonal, presentation components, and toric-boundary components.  Assume
also that every isolated singular value of $\calC(\calA_f)$ is produced
either by ramification of the critical-value map or by the coincidence of
two distinct critical preimages.  Then, with
$d_\Delta=\max_{(a,b)\in\Delta}((a-x_{\min})+(b-y_{\min}))$, one has
$$
\card\Sing_{\mathrm{isol}}\calC(\calA_f)
\leq8d_\Delta^6+8d_\Delta^4-4d_\Delta^3
<8d_\Delta^6+8d_\Delta^4.
$$
\end{theorem}

\begin{proof}
The counting lemma gives
$\card\Sing_{\mathrm{isol}}\calC(\calA_f)
\leq\length(Z_{\mathrm{ram}})+\frac12\length(Z_{\mathrm{off}})$.
The degree calculations give
$\length(Z_{\mathrm{ram}})\leq8d_\Delta^4-4d_\Delta^3$ and
$\length(Z_{\mathrm{off}})\leq16d_\Delta^6$.  Substitution proves the first
inequality.  Since $d_\Delta>0$, discarding the negative term
$-4d_\Delta^3$ gives the strict second inequality.
\end{proof}

 
\begin{remark}
Let $\Delta\subset\R^2$ be a fixed two-dimensional lattice polygon and
let
$\calS_\Delta=\{f\in\C[z^{\pm1},w^{\pm1}]:\Newt(f)=\Delta\}$.  Thus every
vertex of $\Delta$ belongs to $\supp(f)$, while lattice points of
$\Delta$ that are not vertices may or may not occur.  Put
$r=\#\operatorname{Vert}(\Delta)$ and
$M_\Delta=\#(\Delta\cap\Z^2)-r$.  For $f\in\calS_\Delta$, define
$$
q(f)=\#\bigl(\supp(f)\setminus\operatorname{Vert}(\Delta)\bigr).
$$
One cannot assert that adding a nonvertex monomial always creates a
singularity, always destroys a singularity, or always leaves the number
unchanged.  The positions of the exponents, the moduli and phases of the
coefficients, the real solutions of the logarithmic critical equations,
and the way distinct critical branches project under $\Log$ all matter.
The locations of the nonvertex exponents are also essential.  An exponent
in the interior of $\Delta$, an exponent in the relative interior of an
edge, and an exponent close to a vertex enter the logarithmic derivative
equations in different ways.
\end{remark}


 \section*{Example with Triangle Newton polygon}

Consider the lattice triangle
$\Delta=\operatorname{conv}\{(0,0),(4,0),(0,1)\}$.  Its vertex set is
$\operatorname{Vert}(\Delta)=\{(0,0),(4,0),(0,1)\}$.  The polynomial
$g(z,w)=w-1-z^4$ is maximally sparse because
$\supp(g)=\operatorname{Vert}(\Delta)$, and its Newton polygon is exactly
$\Delta$.  We shall prove that the contour of its amoeba has no singular
point.

Now add the two nonvertex monomials $z$ and $z^3$ and consider
$f(z,w)=w-1-z+z^3-z^4$.  Its support is
$\{(0,0),(1,0),(3,0),(4,0),(0,1)\}$, so $\Newt(f)=\Delta$, but $f$ is no
longer maximally sparse.  We shall prove that the contour of its amoeba has
an ordinary transverse node at $(0,\log 2)$.

Thus $g$ and $f$ have the same Newton polygon, the contour of the maximally
sparse polynomial $g$ is nonsingular, and adding nonvertex monomials produces
a singular contour for $f$.

\subsection*{The critical locus of a graph curve}

Let $p$ be a one-variable polynomial and consider the smooth graph curve
$C_p=\{w-p(z)=0\}\subset\T$.  It is smooth because the derivative of
$w-p(z)$ with respect to $w$ is identically one.  The logarithmic Gauss map
of $C_p$ is
$\gamma_p(z)=[-zp'(z):p(z)]$.  On the affine chart where $p(z)\neq0$, define
$h_p(z)=-zp'(z)/p(z)$.  The logarithmic critical locus is
$\gamma_p^{-1}(\mathbb P^1_\R)$, or equivalently
$h_p^{-1}(\R)$ on this affine chart.
Suppose that $p$ has real coefficients, that $z_0\in\R^*$,
$p(z_0)\neq0$, and $h_p'(z_0)\neq0$.  Since $h_p$ is holomorphic and has
real coefficients, it maps the real axis into the real axis.  The condition
$h_p'(z_0)\neq0$ makes it biholomorphic in a neighborhood of $z_0$.  The
inverse image of the real axis is therefore exactly the real axis near
$z_0$.  Consequently, the local logarithmic critical locus is parametrized
by the real variable $t$ through $z=t$ and $w=p(t)$.
On an interval on which $t\neq0$ and $p(t)\neq0$, its critical-value map is
$r_p(t)=(\log|t|,\log|p(t)|)$.  Differentiation gives
$r_p'(t)=(1/t,p'(t)/p(t))$.  In particular, this vector never vanishes,
because its first coordinate is $1/t$.  A singularity of the image can still
arise if two distinct critical parameters have the same logarithmic image
and their image branches cross.

\subsection*{Nonsingularity of the maximally sparse contour}

For $g(z,w)=w-1-z^4$, put $U=z^4$ and $W=-w$.  The curve becomes
$1+U+W=0$.  On logarithmic coordinates, the monomial map is the invertible
real-linear transformation
$(\log|z|,\log|w|)\mapsto(4\log|z|,\log|w|)$.  Hence it preserves whether a
contour germ is smooth, whether two contour values coincide, and whether two
branches meet transversely.
It is therefore enough to study the line $L=\{1+U+W=0\}$.  Its logarithmic
Gauss map is $[U:W]$.  The criticality condition is
$[U:W]\in\mathbb P^1_\R$.  Since $W=-1-U$, this condition forces $U$ to be
real.  Indeed, writing $U=a+ib$ and $W=-1-a-ib$, the condition that $U/W$
be real is $\operatorname{Im}(U\overline W)=0$, and a direct computation
gives $\operatorname{Im}(U\overline W)=-b$.  Thus $b=0$.
The contour of $L$ is consequently parametrized by
$\rho(t)=(\log|t|,\log|1+t|)$, where
$t\in\R\setminus\{-1,0\}$.  Its derivative is
$\rho'(t)=(1/t,1/(1+t))$, which never vanishes.  It remains to prove that
the parametrization has no double value.  Suppose that $\rho(t)=\rho(s)$.
Then $|t|=|s|$, so $s=t$ or $s=-t$.  In the second case,
$|1+t|=|1-t|$.  Squaring both sides gives
$(1+t)^2=(1-t)^2$, hence $4t=0$.  This would imply $t=0$, which is excluded.
Therefore $s=t$, and $\rho$ is injective.
The contour of $L$ is an injectively immersed real curve, so it has no finite
singular point.  The invertible linear change of logarithmic coordinates
then shows that $\calC(\calA_g)$ is nonsingular.

\subsection*{Two distinct critical lifts for the enlarged support}

Write $f(z,w)=w-p(z)$ with
$p(z)=1+z-z^3+z^4$.  The two points
$q_+=(1,2)$ and $q_-=(-1,2)$ belong to $V_f$, because
$p(1)=p(-1)=2$.  They are distinct points of $(\C^*)^2$, but
$\Log(q_+)=\Log(q_-)=(0,\log2)$.

Both points lie in the logarithmic critical locus.  This follows immediately
from the fact that $p$ has real coefficients and both parameters are real.
It can also be checked directly from the logarithmic Gauss map:
$p'(z)=1-3z^2+4z^3$, so $p'(1)=2$ and $p'(-1)=-6$.  Hence
$\gamma_p(1)=[-2:2]=[-1:1]$ and
$\gamma_p(-1)=[-6:2]=[-3:1]$, both of which lie in
$\mathbb P^1_\R$.
We must verify that these are nondegenerate critical lifts.  The affine
Gauss coordinate is $h_p(z)=-zp'(z)/p(z)$, and
$$
h_p'(z)=-\frac{p'(z)+zp''(z)}{p(z)}
+\frac{z(p'(z))^2}{p(z)^2}.
$$
Since $p''(z)=-6z+12z^2$, one has $p''(1)=6$ and $p''(-1)=18$.  Substitution
gives $h_p'(1)=-3$ and $h_p'(-1)=3$.  In particular, neither derivative
vanishes.  The local critical locus near each $q_\pm$ is therefore precisely
the real $z$-axis, and each of the two corresponding local contour branches
is a regular immersed branch.

\subsection*{Classification of the singularity}

For real $t$ near $1$ or $-1$, the contour parametrization is
$r(t)=(\log|t|,\log|p(t)|)$, and
$r'(t)=(1/t,p'(t)/p(t))$.  At the two critical lifts this gives
$r'(1)=(1,1)$ and $r'(-1)=(-1,-3)$.  Their tangent determinant is
$$
\det\begin{pmatrix}1&1\\-1&-3\end{pmatrix}=-2\neq0.
$$
Thus the two regular local branches have distinct tangent lines.  They meet
at the common image $(0,\log2)$ and cross transversely.  By definition, the
contour germ at this value is an ordinary real node.  In particular, it is
not a cusp: neither branch is ramified, since both tangent vectors are
nonzero.  It is not a tacnode, because the tangent determinant is nonzero.

We have proved the following statement.

\begin{proposition}
Let $\Delta=\operatorname{conv}\{(0,0),(4,0),(0,1)\}$.  The maximally sparse
polynomial $g(z,w)=w-1-z^4$ has Newton polygon $\Delta$, and its amoeba
contour has no finite singular point.  The polynomial
$f(z,w)=w-1-z+z^3-z^4$ has the same Newton polygon and is obtained by adding
the nonvertex monomials $z$ and $z^3$.  Its amoeba contour has an ordinary
transverse node at $(0,\log2)$, produced by the two distinct logarithmic
critical points $(1,2)$ and $(-1,2)$.
\end{proposition}

This example also explains why maximal sparsity does not maximize the number
of contour singularities.  For the three vertex monomials, a monomial change
reduces the curve to a line, whose contour is injectively immersed.  The two
additional coefficients create enough freedom for distinct critical lifts to
acquire the same moduli while retaining different tangent directions.  The
node is therefore created by an off-diagonal coincidence rather than by
ramification.

 
\section*{The Newton polygon and its interior lattice point}

Let $\Delta=\operatorname{conv}\{(0,0),(3,0),(0,2)\}$.  
Consider the maximally sparse polynomial
$g(z,w)=w^2-1-z^3$.  Its support is exactly
$\{(0,0),(3,0),(0,2)\}=\operatorname{Vert}(\Delta)$, and hence
$\Newt(g)=\Delta$.  Now add the monomial corresponding to the unique interior
lattice point and define
$f(z,w)=w^2+zw-1-z^3$.  Then
$\supp(f)=\operatorname{Vert}(\Delta)\cup\{(1,1)\}$ and
$\Newt(f)=\Delta$.  We shall prove that the contour of $g$ is nonsingular,
whereas the contour of $f$ has an ordinary transverse node at the origin of
the logarithmic plane.

\subsection*{The maximally sparse contour is nonsingular}

The curve $V_g$ is smooth in $\T$.  Indeed, $g_z=-3z^2$ and $g_w=2w$, and
these two derivatives cannot vanish at a point of $(\C^*)^2$.

Put $U=z^3$ and $W=-w^2$.  The equation $g=0$ becomes $1+U+W=0$.  On
logarithmic coordinates, this monomial substitution induces the invertible
real-linear map
$(\log|z|,\log|w|)\mapsto(3\log|z|,2\log|w|)$.  Consequently, it is enough
to prove that the logarithmic contour of the line
$L=\{1+U+W=0\}\subset(\C^*)^2$ is nonsingular.
The logarithmic Gauss map of $L$ is $[U:W]$.  Its logarithmic critical locus
is therefore characterized by $[U:W]\in\mathbb P^1_\R$.  Since
$W=-1-U$, this forces $U$ to be real.  To see this directly, write
$U=a+ib$.  Then $W=-1-a-ib$, and
$\operatorname{Im}(U\overline W)=-b$.  Thus $U/W$ is real only when $b=0$.
The contour of $L$ is parametrized by
$\rho(t)=(\log|t|,\log|1+t|)$ for
$t\in\R\setminus\{-1,0\}$.  Its derivative is
$\rho'(t)=(1/t,1/(1+t))$, which never vanishes.  The map is also injective.
Indeed, if $\rho(t)=\rho(s)$, then $|t|=|s|$, so either $s=t$ or $s=-t$.
In the latter case, $|1+t|=|1-t|$ implies
$(1+t)^2=(1-t)^2$, hence $t=0$, which is excluded.  Therefore $s=t$.

It follows that the contour of $L$ is an injectively immersed real curve and
has no finite singular point.  Since the logarithmic coordinate change is
invertible, $\calC(\calA_g)$ is also nonsingular.

\subsection*{Smoothness of the polynomial with the interior monomial}

For $f(z,w)=w^2+zw-1-z^3$, one has
$f_z=w-3z^2$ and $f_w=2w+z$.  Suppose that a point of $V_f\cap\T$ were
singular.  The equations $f_z=f_w=0$ would give
$w=3z^2$ and $2w+z=0$.  Since $z\neq0$, these imply $z=-1/6$ and $w=1/12$.
At this point,
$f(-1/6,1/12)=1/144-1/72-1+1/216=-433/432\neq0$.  Therefore no such point
lies on $V_f$, and $V_f$ is smooth in $\T$.

\subsection*{Two critical points with the same logarithmic value}

The points $p_+=(1,1)$ and $p_-=(-1,1)$ lie on $V_f$, because
$f(1,1)=1+1-1-1=0$ and $f(-1,1)=1-1-1+1=0$.  They are distinct, while
$\Log(p_+)=\Log(p_-)=(0,0)$.
The logarithmic Gauss map of $V_f$ is
$\gamma_f(z,w)=[P(z,w):Q(z,w)]$, where
$P=zf_z=zw-3z^3$ and $Q=wf_w=2w^2+zw$.  At the two points one obtains
$\gamma_f(p_+)=[-2:3]$ and $\gamma_f(p_-)=[2:1]$.  Both values belong to
$\mathbb P^1_\R$, so both $p_+$ and $p_-$ are logarithmic critical points.
It remains  to prove that the two critical lifts are nondegenerate
and determine two genuine local contour branches.  Since
$f_w(p_+)=3\neq0$ and $f_w(p_-)=1\neq0$, the holomorphic implicit-function
theorem writes the curve locally as $w=w(z)$.  Implicit differentiation
gives
$w'(z)=-(w-3z^2)/(2w+z)$.  Hence $w'(1)=2/3$ and $w'(-1)=2$.

We use the affine logarithmic Gauss coordinate $h=P/Q$.  Along the curve,
$P'=P_z+P_ww'$ and $Q'=Q_z+Q_ww'$, where
$P_z=w-9z^2$, $P_w=z$, $Q_z=w$, and $Q_w=4w+z$.  At $p_+$ these formulas
give $P=-2$, $Q=3$, $P'=-22/3$, and $Q'=13/3$.  Therefore
$h'(1)=(P'Q-PQ')/Q^2=-40/27\neq0$.  At $p_-$ they give
$P=2$, $Q=1$, $P'=-10$, and $Q'=7$, so
$h'(-1)=-24\neq0$.

Sincee $f$ has real coefficients, the real part of the local curve is sent
by $h$ into $\R$.  Since $h'$ is nonzero at each of the two points, $h$ is a
local biholomorphism there.  The inverse image of $\R$ is therefore exactly
the real local branch.  Thus the logarithmic critical locus is a smooth real
curve near each $p_\pm$, and the real variable $z$ is a valid local
parameter.

\subsection*{The tangent determinant and the node classification}

Along either real critical branch, the logarithmic critical-value map is
$r(z)=(\log|z|,\log|w(z)|)$.  Its tangent vector is
$r'(z)=(1/z,w'(z)/w(z))$.  Since $w=1$ at both points, the two tangent vectors
are $v_+=(1,2/3)$ and $v_-=(-1,2)$.  Their determinant is
$$
\det\begin{pmatrix}1&2/3\\-1&2\end{pmatrix}=\frac83\neq0.
$$
Both local branches are regular because neither tangent vector vanishes, and
the nonzero determinant proves that their tangent lines are distinct.  The
two branches have the same image point $(0,0)$ and meet transversely there.
Consequently, $(0,0)$ is an ordinary transverse node of
$\calC(\calA_f)$.  It is not a cusp, since neither critical-value branch is
ramified, and it is not a tacnode or a higher tangency, since the tangent
determinant is nonzero.

\begin{proposition}
Let $\Delta=\operatorname{conv}\{(0,0),(3,0),(0,2)\}$.  The polygon has the
unique relative-interior lattice point $(1,1)$.  The maximally sparse
polynomial $g(z,w)=w^2-1-z^3$ has Newton polygon $\Delta$, and its amoeba
contour has no finite singular point.  Adding the monomial $zw$ associated
with the interior point gives $f(z,w)=w^2+zw-1-z^3$, whose Newton polygon is
still $\Delta$.  The contour of $f$ has an ordinary transverse node at
$(0,0)$, produced by the two distinct nondegenerate logarithmic critical
points $(1,1)$ and $(-1,1)$.
\end{proposition}

 
\section*{Ramification, saturated two-lift equations, and triangular Newton polygons}

Consider the curve $V=V(f)\subset(\C^*)^2$ defined by
$f(z,w)=w^2+zw-1-z^3$. 
The purpose of this section is
to write the ramification and off-diagonal two-lift systems without any
ambiguity, to record the solutions found by complete numerical computation,
and to give the exact coefficient-incidence conditions for ordinary nodes and
ordinary cusps for polynomials supported in a lattice triangle.

There is an important distinction between an exact algebraic reduction and an
exact certification of all real solutions. The systems below are exact and
have rational coefficients. The displayed solution table is obtained by
high-density numerical computation and Newton refinement. It shows conclusively at the
numerical-algebraic level that $(0,0)$ is not the only singular-value
candidate. A formal theorem asserting that the table exhausts the real locus
still requires a rational univariate representation or interval arithmetic
applied to the saturated ideals. No unperformed Gr\"obner or interval
calculation is claimed in this note.

The derivatives of $f$ are $f_z=w-3z^2$ and $f_w=2w+z$. Put
$P=zf_z=zw-3z^3$ and $Q=wf_w=2w^2+zw$. Introduce independent variables
$\zeta,\eta$ and define
$f^\dagger(\zeta,\eta)=\eta^2+\zeta\eta-1-\zeta^3$,
$P^\dagger=\zeta\eta-3\zeta^3$, and
$Q^\dagger=2\eta^2+\zeta\eta$. The complexified logarithmic critical
equation is
$$
K=(zw-3z^3)(2\eta^2+\zeta\eta)
 -(2w^2+zw)(\zeta\eta-3\zeta^3)=0.
$$
Thus the complexified critical curve is
$\Gamma_f=V(f,f^\dagger,K)\subset(\C^*)^4$. On its real form
$\zeta=\overline z$ and $\eta=\overline w$, the equation $K=0$ is
equivalent to
$\operatorname{Im}(P/Q)=0$, provided $Q\neq0$. The critical-value map is
$\mu=(R,S)$, where $R=z\zeta$ and $S=w\eta$. On the real form,
$R=|z|^2$ and $S=|w|^2$, so the contour map is obtained from $\mu$ by the
local diffeomorphism $(R,S)\mapsto(\frac12\log R,\frac12\log S)$.

Let $F_1=f$, $F_2=f^\dagger$, and $F_3=K$. At a point where
$dF_1,dF_2,dF_3$ are independent, a polynomial tangent vector to
$\Gamma_f$ is the cofactor vector
$$
T_j=(-1)^{j+1}\det
\left(\frac{\partial(F_1,F_2,F_3)}
{\partial(x_1,\ldots,\widehat{x_j},\ldots,x_4)}\right),
\qquad (x_1,x_2,x_3,x_4)=(z,w,\zeta,\eta).
$$
It satisfies $T(F_i)=0$ for $i=1,2,3$. Write
$D=T_z\partial_z+T_w\partial_w+T_\zeta\partial_\zeta+
T_\eta\partial_\eta$. The ramification ideal is therefore
$$
I_{\mathrm{ram}}=(f,f^\dagger,K,D(R),D(S)).
$$
To exclude singular points of $\Gamma_f$, coordinate poles, and components on
which the chosen cofactor vector vanishes identically, this ideal must be
saturated by $zw\zeta\eta$, by the ideal of the $3\times3$ minors of the
Jacobian of $(f,f^\dagger,K)$, and, if necessary, by the common factor of the
four cofactors. A convenient intrinsic notation is
$$
I_{\mathrm{ram}}^{\circ}=
\left(I_{\mathrm{ram}}:(zw\zeta\eta)^\infty\right):
I_{\mathrm{sing}}(\Gamma_f)^\infty.
$$
The real ramification candidates are the points of
$V(I_{\mathrm{ram}}^{\circ})$ satisfying
$\zeta=\overline z$ and $\eta=\overline w$.

The ordinary-cusp test is expressed by the higher tangent derivatives of
$\mu$. At a ramification point $p$, put
$A_2=(D^2R(p),D^2S(p))$ and
$A_3=(D^3R(p),D^3S(p))$. The image germ is an ordinary cusp precisely when
$A_2\neq(0,0)$ and
$$
\det\begin{pmatrix}D^2R(p)&D^2S(p)\\D^3R(p)&D^3S(p)
\end{pmatrix}\neq0.
$$
These conditions are invariant under replacing $T$ by another nonvanishing
local tangent field: the numerical values change, but their required
nonvanishing does not.

For the two-lift construction, we use two copies
$p_i=(z_i,w_i,\zeta_i,\eta_i)$ of the coordinates of $\Gamma_f$. Write
$f_i,f_i^\dagger,K_i,R_i,S_i$ for the corresponding polynomials. The raw
two-lift ideal is
$$
I_2=(f_1,f_1^\dagger,K_1,
f_2,f_2^\dagger,K_2,R_1-R_2,S_1-S_2).
$$
It contains the diagonal component automatically. The diagonal ideal is
$I_{\mathrm{diag}}=(z_1-z_2,w_1-w_2,\zeta_1-\zeta_2,
\eta_1-\eta_2)$. Since $f$ has real coefficients, the real form also has
the automatic conjugate correspondence. In the doubled complexification it
is represented by
$$
I_{\mathrm{conj}}=(z_1-\zeta_2,w_1-\eta_2,
\zeta_1-z_2,\eta_1-w_2).
$$
The ordered, nonzero, off-diagonal and nonconjugate two-lift ideal is
$$
I_{\mathrm{off}}=
\left(
\left(
\left(I_2:I_{\mathrm{diag}}^\infty\right):
I_{\mathrm{conj}}^\infty
\right):(z_1w_1\zeta_1\eta_1z_2w_2\zeta_2\eta_2)^\infty
\right).
$$
If one works in a toric compactification, Cox irrelevant-ideal saturation and
toric-boundary saturation must be carried out separately. Saturation by the Cox
irrelevant ideal removes the nongeometric Cox locus; components contained in
the actual toric boundary are removed by saturation by the boundary
monomial. Working directly in the Laurent ring avoids this distinction.

At a real point of $V(I_{\mathrm{off}})$, the two image branches form an
ordinary node when both restrictions of $d\mu$ are nonzero and their tangent
directions are transverse. With the tangent derivations $D_1,D_2$ on the two
copies, the exact transversality polynomial is
$$
N_{12}=D_1(R_1)D_2(S_2)-D_1(S_1)D_2(R_2).
$$
Thus the ordinary-node locus is obtained by imposing
$N_{12}\neq0$ together with
$(D_iR_i,D_iS_i)\neq(0,0)$ for $i=1,2$. These inequalities are as important
as the equations: without them, ramification coincidences, tangencies, and
higher multiple values remain in the two-lift scheme.

For numerical computation it is useful to write $z=e^{x+i\theta}$ and solve
the quadratic equation in $w$. Its branches are
$$
w_\pm(z)=\frac{-z\pm\sqrt{z^2+4(1+z^3)}}2.
$$
On either branch the critical equation becomes
$G_\pm(x,\theta)=\operatorname{Im}\bigl(z(w_\pm-3z^2)/
(w_\pm(2w_\pm+z))\bigr)=0$. A two-lift solution is found by solving
$G_{\epsilon_1}(x,\theta_1)=G_{\epsilon_2}(x,\theta_2)=0$ together with
$\log|w_{\epsilon_1}|-\log|w_{\epsilon_2}|=0$, while removing equal and
conjugate lifts. Newton refinement gives the following seven transverse
double-value candidates. The two final columns are the slopes of the two
image branches and are unequal in every row.

\begin{center}
\begin{tabular}{rrrr}
\toprule
$X=\log|z|$&$Y=\log|w|$&slope$_1$&slope$_2$\\
\midrule
$-0.264880646688$&$-0.126880403450$&$0.304972653$&$-0.829747055$\\
$-0.084079612491$&$ 0.112124600122$&$-0.883615608$&$0.563198504$\\
$ 0$&$0$&$0.666666667$&$-2.000000031$\\
$ 0.025468741004$&$-0.060326072009$&$4.377475492$&$-2.816140241$\\
$ 0.049197261828$&$ 0.034594036227$&$3.679453196$&$0.739695354$\\
$ 0.121410116315$&$ 0.257555261672$&$2.660725183$&$0.852609120$\\
$ 0.420125138602$&$ 0.404508637398$&$1.225434859$&$2.186414542$\\
\bottomrule
\end{tabular}
\end{center}

At $(X,Y)=(0,0)$ the two critical lifts can be written exactly. They are
$(z,w)=(1,1)$ and $(z,w)=(-1,1)$. Indeed, both satisfy $f=0$, and their
logarithmic Gauss ratios are respectively $-2/3$ and $2$, hence real. The
two displayed slopes are distinct, so the local numerical calculation is
consistent with an ordinary node. %
The remaining six rows are separated from $(0,0)$. Consequently, the complete
continuation calculation says that $(0,0)$ is not the
only contour singularity; it produces six further transverse double-value
candidates.

The ramification computation produces three nonreal-phase candidates, up to
complex conjugation, at approximately
$
(-0.563318453415,\, -0.018466280392),\, 
(0.755165540575,1.369510285746),$
 $(0.788784906144,$ and $1.077673805727).
$
They satisfy the numerical ordinary-cusp jet test. Real-phase stationary
points must not be counted automatically as cusps: the scalar equation
$\operatorname{Im}(P/Q)=0$ vanishes identically along a real branch, and its
Jacobian may become singular there even though the logarithmic image is
immersed. This is precisely why the intrinsic cofactor system
$I_{\mathrm{ram}}^{\circ}$ is preferable to differentiating only the scalar
imaginary-part equation.

The exact conclusion presently justified is therefore twofold. The polynomial
systems whose saturated real solutions give all ramification and nontrivial
two-lift candidates have been written explicitly, and the numerical solution
of those systems shows seven node candidates and three cusp candidates.
 In particular, $(0,0)$ is not the only candidate singular value.

Python was used for the numerical computations, especially NumPy, SciPy, SymPy, and Matplotlib. and 
PGFPlots/LaTeX to produce the figures from the computed data files

 \begin{figure}[ht]
\centering
\includegraphics[width=.25\textwidth]{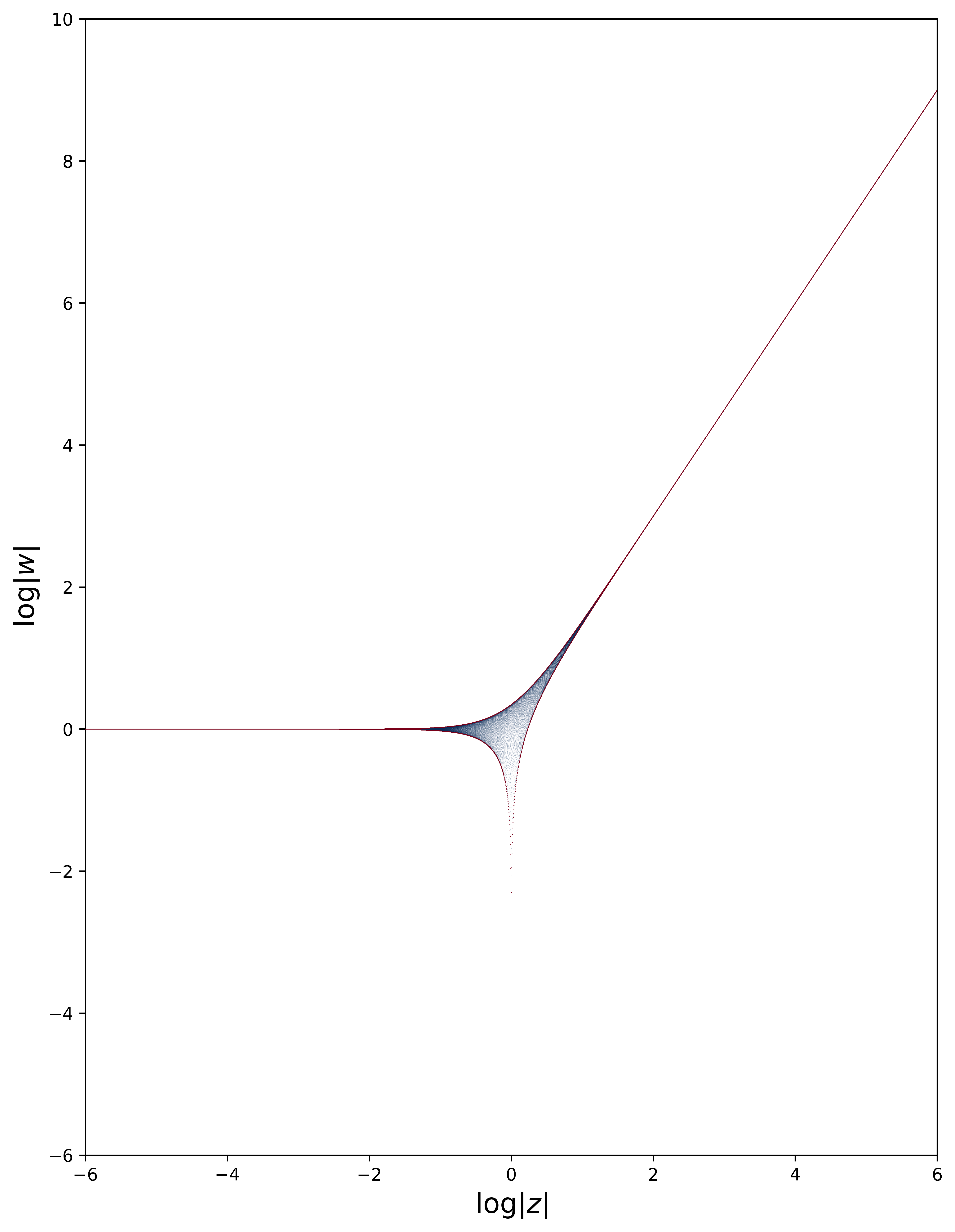} \qquad \includegraphics[width=.25\textwidth]{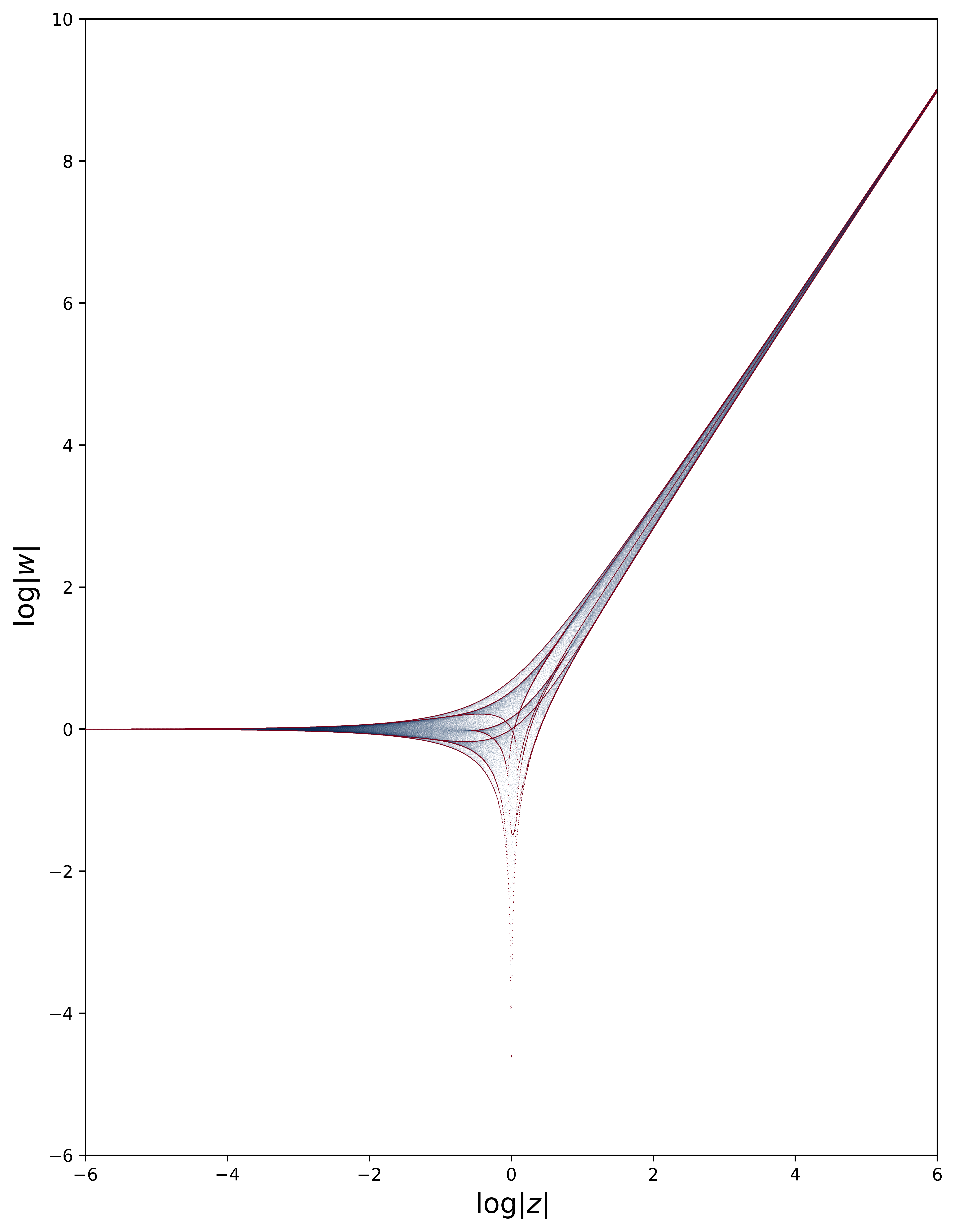}\qquad
 \includegraphics[width=.25\textwidth]{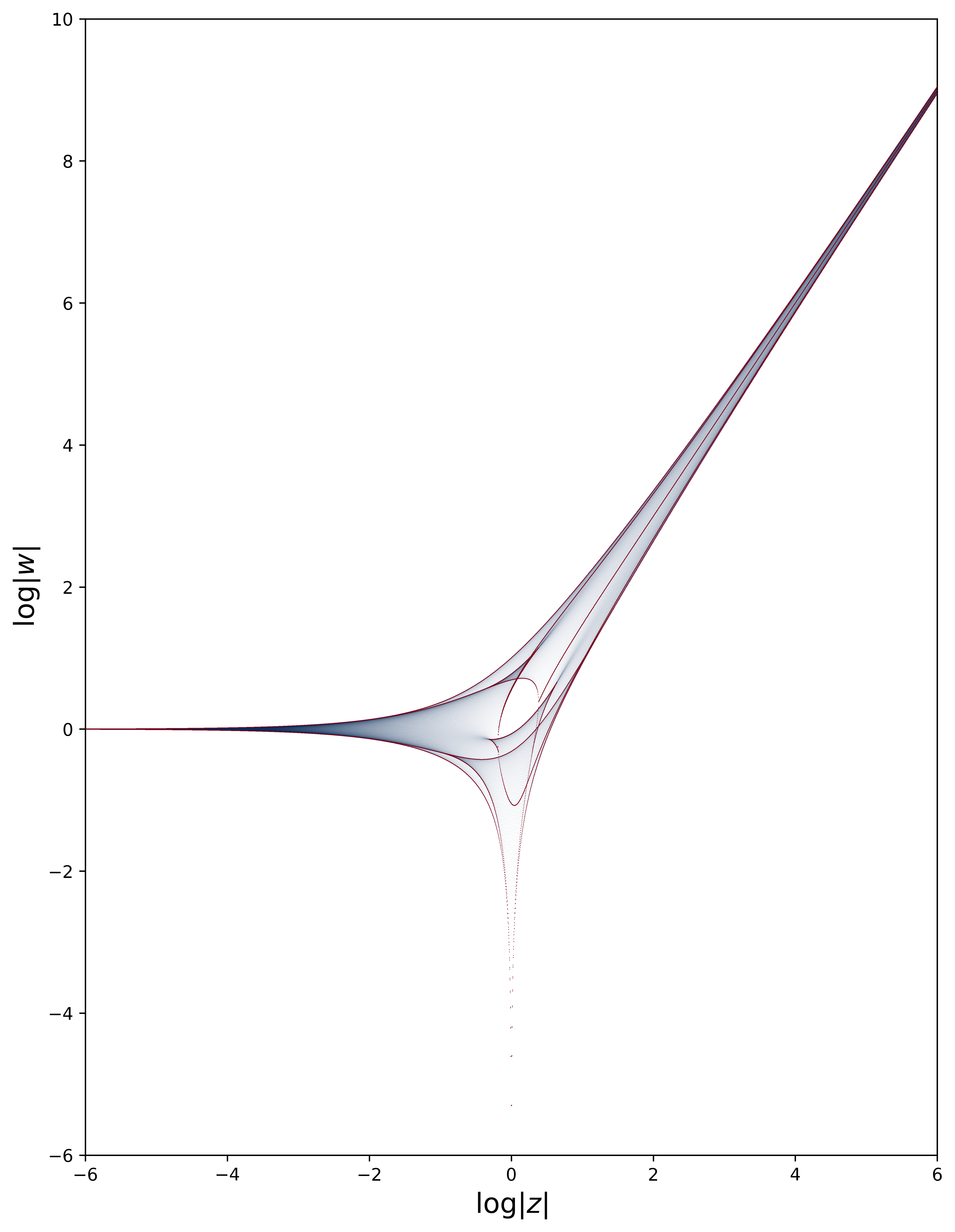} 
\caption{Dark-color amoeba and clear red logarithmic critical values. In the left the solid  amoeba with its smooth contour of the polynomial $f(z,w)= w^2-1-z^3$, the middle  represents the solid amoeba and its singular contour of the polynomial $f(z,w)= w^2+zw-1-z^3$, and the right represents the amoeba and its singular contour of the polynomial $f(z,w)= w^2+4zw-1-z^3$}
\end{figure}

\subsection*{Coefficient conditions for triangular Newton polygons}

Let $\Delta_{m,n}=\operatorname{conv}\{(0,0),(m,0),(0,n)\}$ and let
$
A_{m,n}=\{(i,j)\in\mathbb Z_{\geq0}^2:ni+mj\leq mn\}.
$
Consider the universal polynomial
$F_{\mathbf a}(z,w)=\sum_{(i,j)\in A_{m,n}}a_{ij}z^iw^j$. Its triangle has
$$
I(\Delta_{m,n})=\frac{mn-m-n-\gcd(m,n)+2}{2}
$$
interior lattice points. Hence the triangle contains several interior lattice
points precisely when this integer is at least $2$.
The exact smoothness condition in the torus is that the saturated Jacobian
ideal
$
(F_{\mathbf a},zF_{\mathbf a,z},wF_{\mathbf a,w}):(zw)^\infty
$
be the unit ideal after specialization of the coefficients. Equivalently, the
sparse discriminant of the support does not vanish, with the usual additional
face-nondegeneracy conditions when the principal $A$-determinant is used.
This condition guarantees that singularities under discussion belong to the
contour map and are not inherited from a singular algebraic curve.

Define $P_{\mathbf a}=zF_{\mathbf a,z}$,
$Q_{\mathbf a}=wF_{\mathbf a,w}$ and introduce an independent coefficient
vector $\mathbf b$ for the dagger polynomial. Put
$
K_{\mathbf a,\mathbf b}=
P_{\mathbf a}Q_{\mathbf b}^\dagger-
Q_{\mathbf a}P_{\mathbf b}^\dagger.
$
The universal critical ideal is
$I_\Gamma=(F_{\mathbf a},F_{\mathbf b}^\dagger,
K_{\mathbf a,\mathbf b})$. For real coefficient data one specializes
$\mathbf b=\overline{\mathbf a}$; for real coefficients this is simply
$\mathbf b=\mathbf a$.

Let $T$ and $D$ be the universal cofactor tangent field constructed from
$I_\Gamma$. The universal cusp incidence ideal is
$$
\mathcal I_{\mathrm{cusp}}=
(I_\Gamma,D(R),D(S)):
\bigl(zw\zeta\eta\,I_{\mathrm{sing}}(\Gamma)\bigr)^\infty.
$$
An exact coefficient condition guaranteeing at least one real ordinary cusp
is the existence of a point in the real form of
$V(\mathcal I_{\mathrm{cusp}})$ at which
$(D^2R,D^2S)\neq(0,0)$ and
$
C_3=D^2R\,D^3S-D^2S\,D^3R\neq0.
$
This is a finite system of polynomial equalities, polynomial inequalities, and
complex-conjugation equations in the coefficients and the lift variables. It
is therefore an exact semialgebraic condition, not a heuristic genericity
statement. Eliminating the lift variables gives the cusp discriminant ideal in
coefficient space; retaining the nonvanishing factors selects the open stratum
of ordinary cusps from its higher-degeneracy locus.

For nodes, take two universal copies of $I_\Gamma$, impose
$R_1-R_2=S_1-S_2=0$, and saturate by the diagonal, the conjugate correspondence,
the torus monomial, and the singular-locus ideals of the two copies. Denote the
result by $\mathcal I_{\mathrm{node}}$. An exact coefficient condition
guaranteeing at least one real ordinary node is the existence of a real-form
solution of $\mathcal I_{\mathrm{node}}$ satisfying
$$
(D_iR_i,D_iS_i)\neq(0,0),\qquad i=1,2,
$$
and
$
D_1R_1D_2S_2-D_1S_1D_2R_2\neq0.
$
The equalities guarantee two distinct nonconjugate critical lifts with the same
logarithmic value; the first inequalities exclude ramification at either lift;
the determinant inequality guarantees transverse image tangents. 
The projections of the cusp and node incidence schemes to coefficient space
are constructible. Their Zariski closures are obtained from the elimination
ideals
$$
\mathfrak D_{\mathrm{cusp}}=
\mathcal I_{\mathrm{cusp}}\cap\C[\mathbf a,\mathbf b],
\qquad
\mathfrak D_{\mathrm{node}}=
\mathcal I_{\mathrm{node}}\cap\C[\mathbf a,\mathbf b].
$$
Vanishing of these elimination ideals alone is not sufficient to guarantee a
real ordinary singularity, because elimination also retains complex lifts and
higher degeneracies. The exact guarantee is obtained only after imposing the
real-form equations and the displayed nonvanishing conditions. 

If the coefficients depend algebraically on a real parameter $t$, creation or
annihilation of nodes and cusps occurs when the coefficient path meets the
corresponding incidence projection. A transverse meeting of the ordinary cusp
stratum, with $C_3\neq0$ and with the parameter derivative transverse to the
elimination hypersurface, gives the stable one-parameter cusp transition. A
transverse meeting of the ordinary two-lift stratum with nonzero tangent
determinant gives a stable transverse node.


\section{A Nontriangular Newton Polygon with an Interior-Monomial Contour Node}

Let $\Delta=[0,2]\times[0,2]$.  
Consider
$g(z,w)=1+z^2+w^2-2z^2w^2$.  Its support is exactly
$\{(0,0),(2,0),(0,2),$ \, $(2,2)\}=\operatorname{Vert}(\Delta)$, so $g$ is
maximally sparse and $\Newt(g)=\Delta$.  We shall prove that the contour of
its amoeba has no singular point.
Now add nonvertex monomials and put
$f(z,w)=1+z^2+w^2-2z^2w^2-w+z-zw$.  The monomials $z$ and $w$ correspond to
the nonvertex boundary points $(1,0)$ and $(0,1)$, while $zw$ corresponds to
the unique relative-interior lattice point $(1,1)$.  Hence
$\supp(f)=\operatorname{Vert}(\Delta)\cup\{(1,0),(0,1),(1,1)\}$ and
$\Newt(f)=\Delta$.  We shall prove that $V_f$ is smooth and that
$\calC(\calA_f)$ has an ordinary transverse node at $(0,0)$.

\subsection*{The contour of the maximally sparse polynomial}

The curve $V_g$ is smooth in $\T$.  Indeed,
$g_z=2z(1-2w^2)$ and $g_w=2w(1-2z^2)$.  If both derivatives vanished in
$\T$, then $z^2=w^2=1/2$, but substitution would give $g=3/2\neq0$.
Set $U=z^2$ and $V=w^2$.  The equation $g=0$ becomes
$G(U,V)=1+U+V-2UV=0$.  The monomial map
$(z,w)\mapsto(U,V)=(z^2,w^2)$ is a local biholomorphism of $\T$, and on
logarithmic coordinates it induces the invertible real-linear map
$(x,y)\mapsto(2x,2y)$.  It consequently identifies the logarithmic critical
locus of $V_g$ with the inverse image of the logarithmic critical locus of
$V_G$, and it preserves regular contour germs, coincidences, and transverse
intersections.  It is therefore enough to prove that the contour of $G$ is
nonsingular.

The curve $V_G$ is the graph
$V=-(1+U)/(1-2U)$.  Its logarithmic Gauss map is
$[U(1-2V):V(1-2U)]$.  On the curve, an affine coordinate for this map is
$h(U)=-3U/((1+U)(1-2U))$.  Write $U=x+iy$.  Apart from its zeros and poles,
the imaginary part of $U/((1+U)(1-2U))$ has the same vanishing set as
$y(1+2|U|^2)$.  Indeed, if
$D=(1+U)(1-2U)=1-U-2U^2$, then
$\operatorname{Im}(U\overline D)=y(1+2|U|^2)$.  Since
$1+2|U|^2>0$, the condition $h(U)\in\R$ forces $y=0$.  Thus the entire
logarithmic critical locus is parametrized by real
$t\in\R\setminus\{-1,0,1/2\}$.

The contour parametrization is
$\rho(t)=(\log|t|,\log|(1+t)/(1-2t)|)$.  Its derivative is
$\rho'(t)=(1/t,1/(1+t)+2/(1-2t))$.  The first coordinate never vanishes, so
$\rho$ is an immersion.

It remains to verify global injectivity.  If $\rho(t)=\rho(s)$, then
$|t|=|s|$, and therefore $s=t$ or $s=-t$.  In the second case one would have
$|(1+t)/(1-2t)|=|(1-t)/(1+2t)|$.  Squaring and clearing the nonzero
denominators gives
$(1+t)^2(1+2t)^2=(1-t)^2(1-2t)^2$.  The difference between the two sides
factors as $12t(1+2t^2)$.  For real $t$, it vanishes only at $t=0$, which is
not in the parameter domain.  Hence $s=t$, and $\rho$ is injective.
The contour of $G$ is therefore an injectively immersed real curve and has
no finite singular point.  The invertible logarithmic coordinate change
implies that $\calC(\calA_g)$ is nonsingular.

\subsection*{Smoothness for the enlarged polynomial}

  For $f(z,w)=1+z^2+w^2-2z^2w^2-w+z-zw$, differentiation gives $f_z=-4w^2z-w+2z+1$ and $f_w=-4wz^2+2w-z-1$.  Define
$
A={} -\dfrac{136w^2+52wz-76w+8z-203}{140},\,\,
B={}\dfrac{52wz^2+68wz+8w+8z^2-199z-65}{280},$ and 
$C={}\dfrac{136w^3-144w^2-59w+61}{280}.
$
Direct expansion gives the exact polynomial identity $Af+Bf_z+Cf_w=1$.  If $f=f_z=f_w=0$ had a common solution in $\mathbb C^2$, substitution into this identity would give $0=1$.  Therefore $V_f$ is smooth in the whole affine plane and, in particular, in $(\mathbb C^*)^2$.

\subsection{Determination of every critical lift above the origin}

 A point above $(0,0)$ satisfies $|z|=|w|=1$.  On this unit torus, $\overline z=z^{-1}$ and $\overline w=w^{-1}$.  Regarding $f$ as a quadratic polynomial in $w$, one has
$f(z,w)=(1-2z^2)w^2-(1+z)w+(1+z+z^2).$
 Multiplying the conjugate equation by $z^2w^2$ gives
$z^2w^2\overline f=(z^2+z+1)w^2-(z^2+z)w+(z^2-2).$
 Eliminating $w$ from these two equations yields the exact resultant
$\operatorname{Res}_w(f,z^2w^2\overline f)=7(z-1)^2(z+1)^4(z^2-z+1).$
  Hence a point of $V_f$ in the fiber $\operatorname{Log}^{-1}(0,0)$ must come from $z=1$, $z=-1$, or $z^2-z+1=0$.  Substitution gives the three real points $p_+=(1,1)$, $p_-=(-1,1)$, and $p_0=(-1,-1)$, together with the conjugate pair $(e^{i\pi/3},e^{2i\pi/3})$ and $(e^{-i\pi/3},e^{-2i\pi/3})$.
  Put $P=zf_z$, $Q=wf_w$, and $H=\operatorname{Im}(P\overline Q)$.  The logarithmic criticality condition is $H=0$.  At the three real points, the values of $[P:Q]$ are
$\gamma_f(p_+)=[-2:-4]=[1:2],\, 
\gamma_f(p_-)=[-2:-2]=[1:1],$
$\gamma_f(p_0)=[-4:-2]=[2:1].$
  So all three are logarithmic critical points.  At the two nonreal points coming from $z^2-z+1=0$, direct substitution gives $H=-21\sqrt3/2$ and $H=21\sqrt3/2$.  They are not logarithmic critical. Therefore,
$\operatorname{Crit}(\operatorname{Log}|_{V_f})\cap\operatorname{Log}^{-1}(0,0)=\{p_+,p_-,p_0\}.$

\subsection{Nondegeneracy of the three critical lifts}

\noindent Since $f_w(p_+)=-4$, $f_w(p_-)=-2$, and $f_w(p_0)=2$, the implicit-function theorem writes $V_f$ locally as a holomorphic graph $w=w(z)$ at all three points.  Implicit differentiation gives $w'=-f_z/f_w$, hence

$$w'(1)=-\frac12\quad\text{at }p_+,\qquad
w'(-1)=1\quad\text{at }p_-,\qquad
w'(-1)=-2\quad\text{at }p_0.$$

\noindent Use the affine Gauss coordinate $h=P/Q$.  Along $V_f$, its derivative is $h'=(P'Q-PQ')/Q^2$, where $P'=P_z+P_ww'$ and $Q'=Q_z+Q_ww'$.  The required partial derivatives are $P_z=-8w^2z-w+4z+1$, $P_w=-8wz^2-z$, $Q_z=-8w^2z-w$, and $Q_w=-8wz^2+4w-z-1$.

\medskip
\noindent At $p_+$, one obtains $P=-2$, $Q=-4$, $P'=1/2$, and $Q'=-6$, so $h'(1)=-7/8\ne0$.  At $p_-$, one obtains $P=-2$, $Q=-2$, $P'=-3$, and $Q'=3$, so $h'(-1)=3\ne0$.  At the omitted point $p_0$, one has
$P=-4,\, Q=-2,\, P_z=6,\, P_w=9,\, Q_z=9,\,  Q_w=4.$
Since $w'=-2$ there, it follows that $P'=6+9(-2)=-12$ and $Q'=9+4(-2)=1$.  Therefore $h'(-1)=((-12)(-2)-(-4)(1))/4=7\ne0$.  Thus the logarithmic Gauss map is locally biholomorphic at every one of the three real critical points.  Since $f$ and $h$ have real coefficients, the inverse image of $\mathbb P^1_{\mathbb R}$ is locally the real branch through each point.  All three critical lifts are therefore regular and nondegenerate.

\subsection{Classification of the common critical value}

\noindent Along each real critical branch, use the real coordinate $z$ and write the critical-value map as $r(z)=(\log|z|,\log|w(z)|)$.  Its tangent vector is $r'(z)=(1/z,w'(z)/w(z))$.  At the three lifts this gives
$v_+=\Big(1,-\frac12\Big),\,
v_-=(-1,1),\, 
v_0=(-1,2).$
 None of these vectors vanishes, so none of the three image branches is ramified.  Their pairwise determinants are
$\di \det(v_+,v_-)=\frac12,\,
\det(v_+,v_0)=\frac32,\,
\det(v_-,v_0)=-1.$

\noindent Every determinant is nonzero.  The three contour branches are therefore smooth and pairwise transverse, with respective slopes $-1/2$, $-1$, and $-2$.  Since the complcete critical-fiber calculation proves that exactly these three branches pass through $(0,0)$, the origin is an ordinary transverse triple point of the contour.  It is not a node, because a node has exactly two smooth transverse branches.  It is not a cusp, since none of the three branches is ramified, and it is not a tacnode, since no two tangent lines coincide.

\begin{proposition}
Let $\Delta=[0,2]^2$.  This polygon is nontriangular and has the unique relative-interior lattice point $(1,1)$.  The maximally sparse polynomial $g(z,w)=1+z^2+w^2-2z^2w^2$ has Newton polygon $\Delta$.  Suppose, as established separately, that its amoeba contour has no finite singular point.  Adding the boundary monomials $z$ and $w$ and the interior monomial $zw$ gives the polynomial $f(z,w)=1+z^2+w^2-2z^2w^2-w+z-zw$, with the same Newton polygon.  The polynomial $f$ is smooth in $(\mathbb C^*)^2$.  Its amoeba contour has an ordinary transverse triple point at $(0,0)$, produced by the three distinct nondegenerate logarithmic critical points $(1,1)$, $(-1,1)$, and $(-1,-1)$.  The three contour branches have respective tangent slopes $-1/2$, $-1$, and $-2$.
\end{proposition}

\bigskip

 
 \begin{figure}[ht]
\centering
\includegraphics[width=.25\textwidth]{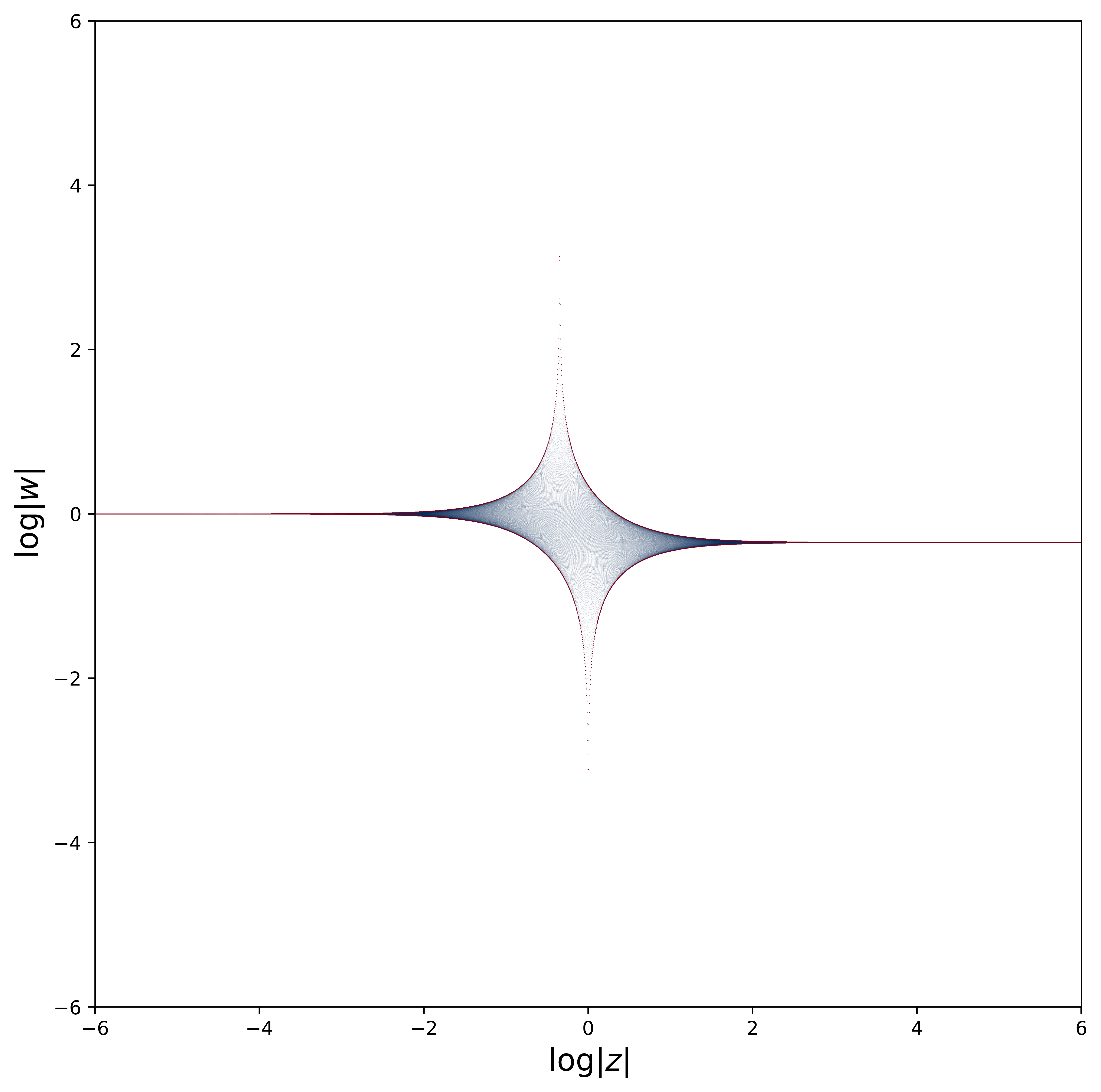} \qquad \includegraphics[width=.25\textwidth]{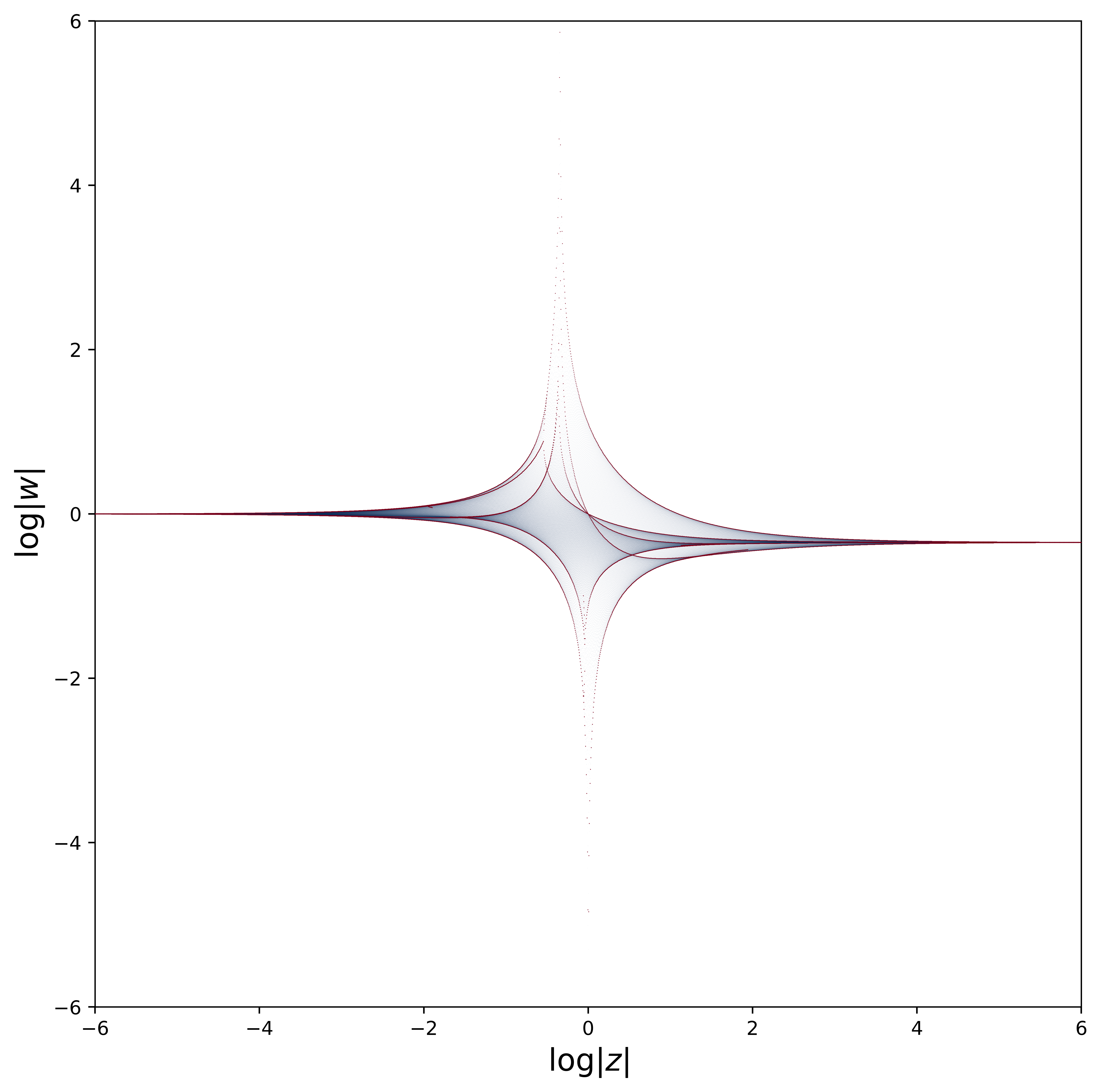}
\caption{Dark-color amoeba and clear red logarithmic critical values. In the left the solid  amoeba with its smooth contour of the polynomial $f(z,w)=1+z^2+w^2-2z^2w^2$, and the right represents the amoeba and its singular contour of the polynomial $f(z,w)=1+z^2+w^2-2z^2w^2-w+z-zw$.}
\end{figure}

 \begin{figure}[ht]
\centering
\includegraphics[width=.25\textwidth]{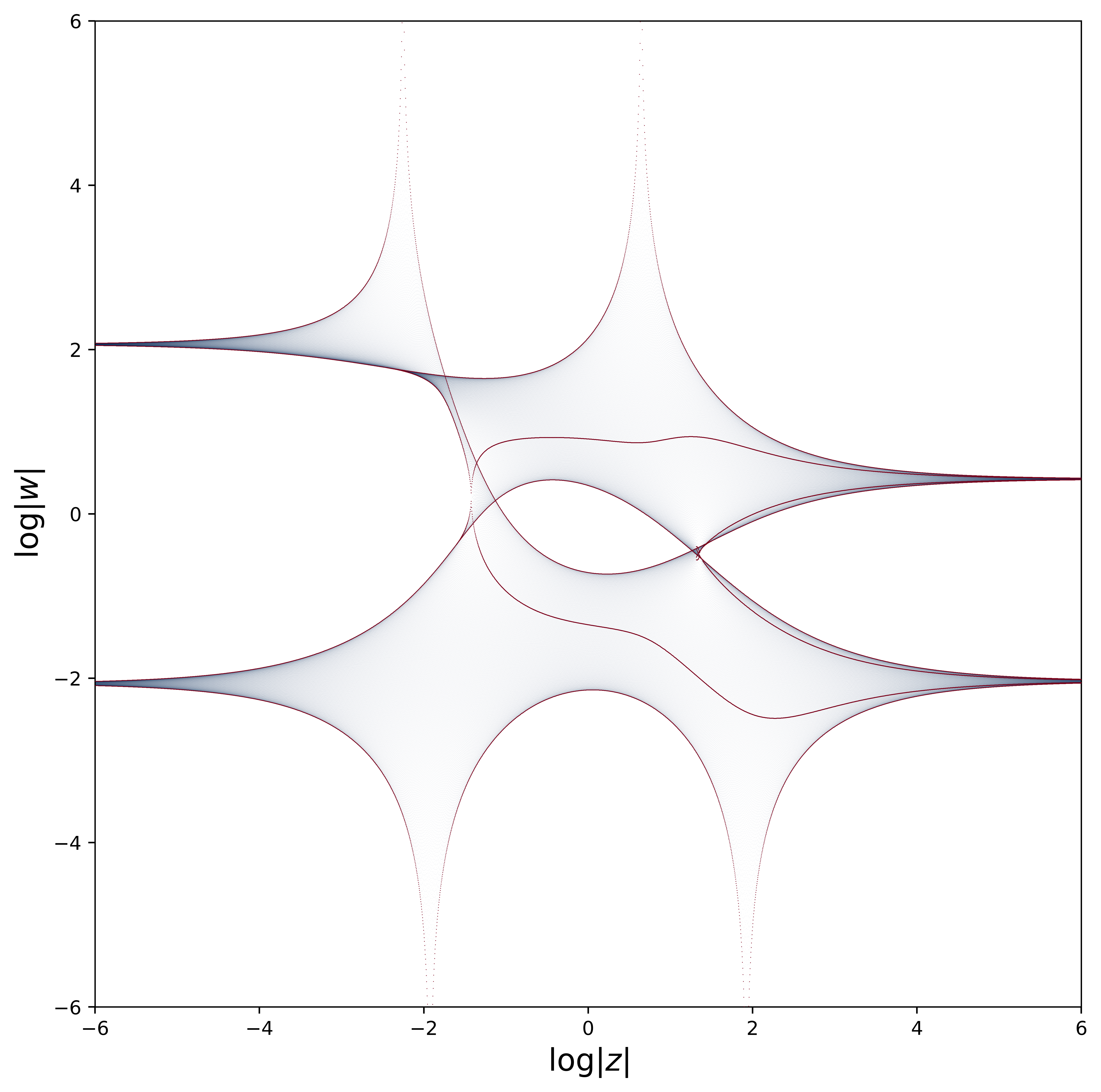}
 \qquad \includegraphics[width=.25\textwidth]{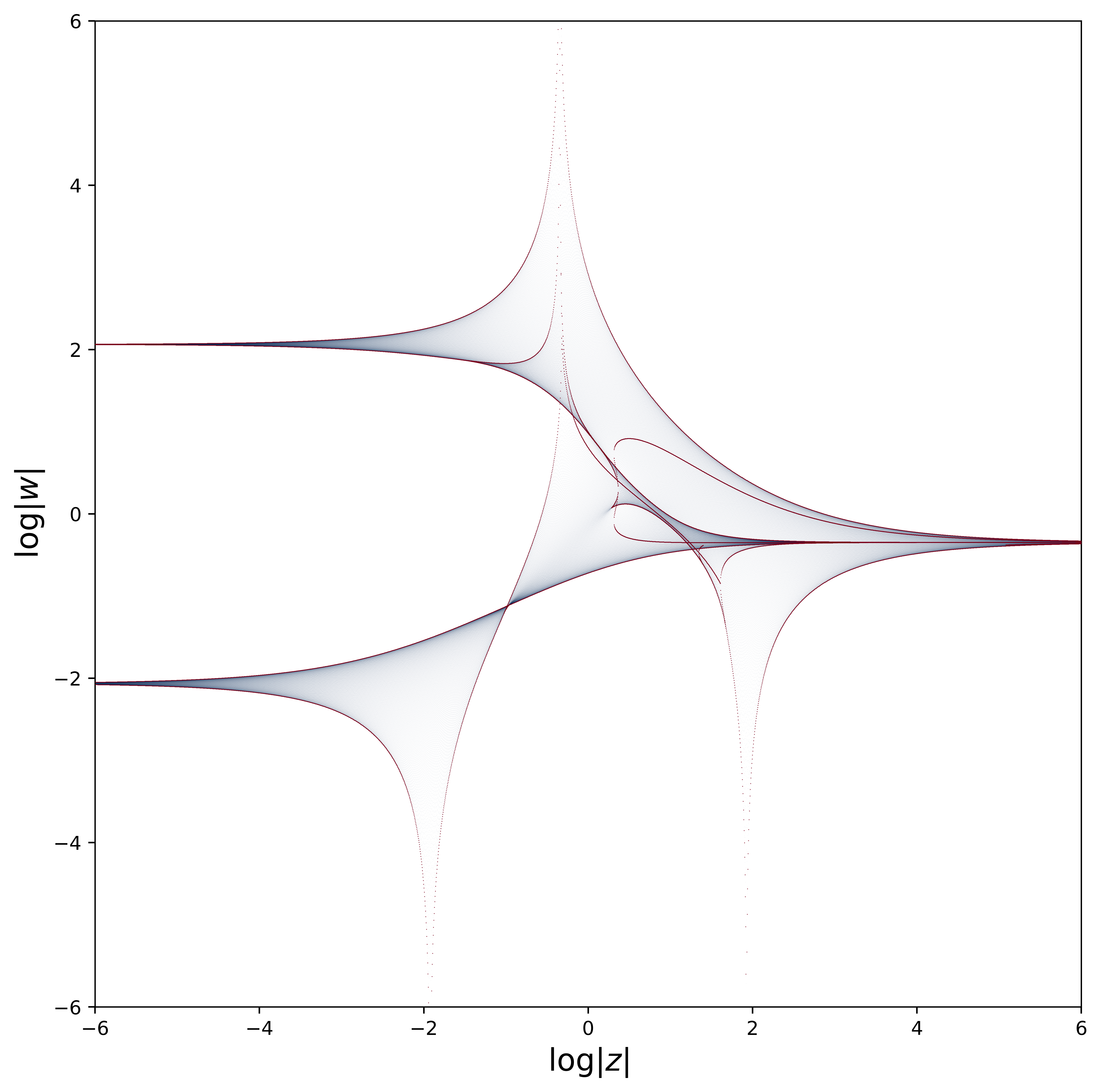}
\caption{Dark-color amoeba and clear red logarithmic critical values. In the left the solid  amoeba with its smooth contour of the polynomial $f(z,w)=1+z^2+w^2-5z^2w^2-8w+7z+27zw-9zw^2-7z^2w$, and the right represents the amoeba and its singular contour of the polynomial $f(z,w)=1+z^2+w^2-2z^2w^2-8w+7z-10zw$.}
\end{figure}


\section{A Bound for the Number of Cusps of an Amoeba Contour}
 
Let $\Delta\subset\mathbb R^2$ be a two-dimensional lattice polygon and let
$f\in\mathbb C[z^{\pm1},w^{\pm1}]$ have Newton polygon $\Delta$.  Write
$C_f=\{f=0\}\subset(\mathbb C^*)^2$, and denote its toric closure in
$X_\Delta$ by $\overline C_f$.  The critical locus  of the logarithmic map on 
 $C_f$ is denoted by $S_f$, and the contour is $\mathcal K_f = {\mathcal C}\mathscr{A}_f =\Log(S_f)$.

Let the set $\mathscr C_\Delta$ consisting of those $C_f$ for which $C_f$ is smooth,
$S_f$ is a smooth real one-dimensional manifold, and $\overline C_f$ is a
smooth torically nondegenerate curve meeting every toric boundary divisor
transversely and away from the zero-dimensional torus orbits.  Smoothness of
$S_f$ does not by itself imply that the restriction
$\Log|_{S_f}:S_f\to\mathbb R^2$ has only isolated ramification points.
Consequently, a finite numerical statement must concern isolated cusps, or
must include the additional hypothesis that this restriction has only
isolated ramification.  Let $\kappa(f)$ denote the number of distinct isolated
cusp values of $\mathcal K_f$.  Multiple cusp lifts with the same logarithmic
image are counted only once.  Counting lifts instead would give a number at
least as large, so every upper bound proved below also bounds $\kappa(f)$.

Introduce the normalized area, the number of interior lattice points, and the
number of boundary lattice points by
$n=\VolZ(\Delta)=2\Area(\Delta)$, $g=|\operatorname{int}(\Delta)\cap
\mathbb Z^2|$, and $b=|\partial\Delta\cap\mathbb Z^2|$.  The two coordinate
widths are
$\omega_z=\max_{(a,c)\in\Delta}c-\min_{(a,c)\in\Delta}c$ and
$\omega_w=\max_{(a,c)\in\Delta}a-\min_{(a,c)\in\Delta}a$.  Thus
$\omega_z$ is the degree of $z:\overline C_f\to\mathbb P^1$, while
$\omega_w$ is the degree of $w:\overline C_f\to\mathbb P^1$.

\begin{theorem}
For every $C_f\in\mathscr C_\Delta$ whose contour cusps are isolated, one has
$\kappa(f)\leq 2n(2n-b)+4n(\omega_z+\omega_w)$.  In particular,
$\kappa(f)\leq12n^2-2bn\leq12n^2-6n<12n^2=48\Area(\Delta)^2$.

If, on every irreducible component of the complexified logarithmic critical
curve introduced below, neither of the functions $z\zeta$ and $w\eta$ is
constant, then the sharper estimate
$\kappa(f)\leq2n(2n-b)+4n\min\{\omega_z,\omega_w\}$ holds.  It implies
$\kappa(f)\leq8n^2-2bn\leq8n^2-6n<8n^2=32\Area(\Delta)^2$.
\end{theorem}

\begin{proof}
The logarithmic Gauss map of $C_f$ is
$\gamma_f:\overline C_f\to\mathbb P^1$, with
$\gamma_f(z,w)=[zf_z:wf_w]$.  Toric nondegeneracy implies that its degree is
$\deg\gamma_f=n$.  The genus of $\overline C_f$ is $g$.  Riemann--Hurwitz
therefore gives the degree of the ramification divisor of $\gamma_f$ as
$R_\gamma=2g-2+2n$.

Let $\overline C_f^\sigma$ be the coefficient-conjugate curve, with
coordinates $(\zeta,\eta)$, and let $\gamma_f^\sigma$ be its logarithmic
Gauss map.  Consider the fiber product
$\Gamma=\overline C_f\mathbin{\times}_{\mathbb P^1}
\overline C_f^\sigma$, defined by
$f(z,w)=0$, $\overline f(\zeta,\eta)=0$, and
$\gamma_f(z,w)=\gamma_f^\sigma(\zeta,\eta)$.  Let
$\nu:\widetilde\Gamma\to\Gamma$ be its normalization.  The real fixed locus
specified by $(\zeta,\eta)=(\overline z,\overline w)$ maps onto $S_f$.
Thus every cusp lift is represented on $\widetilde\Gamma$.

Each projection from $\widetilde\Gamma$ to one of the two curve factors has
total degree $n$.  Write the connected components of
$\widetilde\Gamma$ as $D_j$, with genera $g_j$ and projection degrees $d_j$;
then $\sum_jd_j=n$.  Branching of the first projection can occur only above
branch values of the second logarithmic Gauss map.  Its total ramification is
therefore at most $nR_\gamma$.  Applying Riemann--Hurwitz component by
component gives
$\sum_j(2g_j-2)\leq n(2g-2)+n(2g-2+2n)=2n(2g-2+n)$.
This estimate remains valid when the fiber product is reducible; normalization
can only separate its branches.

On $\widetilde\Gamma$ define $Z=z\zeta$ and $W=w\eta$.  Along the real fixed
locus these satisfy $Z=|z|^2$, $W=|w|^2$, and hence
$\log Z=2\log|z|$, $\log W=2\log|w|$.  At a cusp lift the derivative of the
critical-value map vanishes.  Equivalently, both logarithmic differentials
$d\log Z=dZ/Z$ and $d\log W=dW/W$ vanish there.

The function $z$ has degree $\omega_z$ on each curve factor.  Since the two
projections have total degree $n$, the zero divisor and the pole divisor of
$Z=z\zeta$ each have degree at most $2n\omega_z$.  The differential
$d\log Z$ has at most a simple pole at every zero or pole of $Z$, so its pole
divisor has degree at most $4n\omega_z$.  The corresponding estimate for
$d\log W$ is $4n\omega_w$.

For a nonzero meromorphic differential $\alpha$ on a smooth compact curve
$D_j$, the canonical-divisor formula gives
$\deg\operatorname{Zeros}(\alpha)=2g_j-2+deg\operatorname{Poles}(\alpha)$.
On each $D_j$ at least one of $d\log Z$ and $d\log W$ is nonzero whenever the
critical-value map is not constant there.  Assign to $D_j$ one such
differential.  Every isolated cusp lift is a zero of the assigned
differential.  Summing the canonical-divisor formula and using both possible
coordinate differentials yields
$\kappa(f)\leq2n(2g-2+n)+4n(\omega_z+\omega_w)$.

Pick's formula says $n=2g+b-2$, and hence $2g-2=n-b$.  Substitution gives
$2n(2g-2+n)=2n(2n-b)$, proving the first displayed bound.  Each coordinate
width of a two-dimensional lattice polygon is at most its normalized area:
$\omega_z\leq n$ and $\omega_w\leq n$.  For example, after choosing a
lattice direction of the indicated width, a primitive transverse lattice
segment together with the extremal supporting lines gives a lattice triangle
of normalized area at least that width inside the corresponding strip; the
polygon has no smaller normalized area.  Since $b\geq3$, the area-only
inequalities follow.

Under the additional nonconstancy condition, the same coordinate
differential may be chosen on every component.  Choosing the coordinate with
smaller width replaces $4n(\omega_z+\omega_w)$ by
$4n\min\{\omega_z,\omega_w\}$.  The sharper inequalities then follow from
$\min\{\omega_z,\omega_w\}\leq n$ and $b\geq3$.
\end{proof}


\section{An Improved Degree Bound for Cusps of Amoeba Contours}
 
Let $f(z,w)$ define a smooth affine plane curve of projective degree $d$, and
let $\Delta$ be its Newton polygon.  After multiplication by a monomial if
necessary, the exponent vectors of $f$ lie in the standard degree triangle
$d\Delta_2=\operatorname{conv}\{(0,0),(d,0),(0,d)\}$.  Assume that the curve
belongs to the class $\mathscr C_\Delta$ described in the question and that
the ramification points of the critical-value map are isolated.  Denote by
$\kappa(f)$ the number of distinct cusp values of the contour of the amoeba.

The previously obtained polygonal estimate was
$\kappa(f)\leq 2n(2n-b)+4n(\omega_z+\omega_w)$, where
$n=\VolZ(\Delta)=2\Area(\Delta)$,
$b=|\partial\Delta\cap\mathbb Z^2|$, and $\omega_z,\omega_w$ are the two
coordinate lattice widths of $\Delta$.  The estimate
$\kappa(f)<12d^4$ followed by using $n\leq d^2$ and then replacing both
widths by $n$.  That replacement loses substantial information: a polygon
contained in $d\Delta_2$ satisfies $\omega_z\leq d$ and
$\omega_w\leq d$, not merely $\omega_z,\omega_w\leq d^2$.

\begin{theorem}
Under the hypotheses above, the constant $12$ can be replaced universally by
$13/2$.  More precisely,
$\kappa(f)\leq4d^4+8d^3-6d^2\leq\frac{13}{2}d^4$.
For $d\geq3$ one has the stronger uniform estimate
$\kappa(f)\leq6d^4$, and asymptotically one has
$\kappa(f)\leq(4+O(d^{-1}))d^4$.
\end{theorem}

\begin{proof}
Expanding the polygonal bound gives
$\kappa(f)\leq4n^2-2bn+4n(\omega_z+\omega_w)$.  Since
$\Delta\subset d\Delta_2$, monotonicity of Euclidean area gives $n\leq d^2$.
Projection of $d\Delta_2$ onto either coordinate axis has length $d$;
therefore $\omega_z+\omega_w\leq2d$.  Every two-dimensional lattice polygon
has at least three boundary lattice points, so $b\geq3$.  Consequently,
$\kappa(f)\leq4n^2+(8d-6)n$.

For every fixed $d\geq1$, the function $x\mapsto4x^2+(8d-6)x$ is increasing
for $x\geq0$.  Inserting $n\leq d^2$ therefore yields
$\kappa(f)\leq4d^4+8d^3-6d^2$.  After division by $d^4$, the coefficient is
$q(d)=4+8/d-6/d^2$.  For integral $d\geq1$, its largest value occurs at
$d=2$, where $q(2)=13/2$.  Indeed, $q(1)=6$, and $q$ is decreasing for
$d\geq2$ because $q'(d)=(-8d+12)/d^3<0$ there.  This proves the universal
constant $13/2$.
For $d\geq3$, monotonicity gives $q(d)\leq q(3)=6$, proving
$\kappa(f)\leq6d^4$.  Finally, the exact expression
$q(d)=4+8/d-6/d^2$ tends to $4$, which proves the stated asymptotic estimate.
\end{proof}

The improvement is particularly clear for a full-support (i.e. dense) degree-$d$ polynomial. 
In that case $\Delta=d\Delta_2$, so $n=d^2$, $b=3d$, and
$\omega_z=\omega_w=d$.  Direct substitution into the polygonal estimate gives
$\kappa(f)\leq4d^4+2d^3$.  Thus, for the full degree triangle, the coefficient
is $4+2/d$ rather than $12$.  In particular, it is at most $5$ for $d\geq2$
and converges to $4$ as $d$ tends to infinity.

There is a further improvement when neither logarithmic coordinate function
is constant on any irreducible component of the normalized complexified
critical curve.  Under that additional condition, the sharper polygonal
estimate is
$\kappa(f)\leq2n(2n-b)+4n\min\{\omega_z,\omega_w\}$.  Using
$n\leq d^2$, $b\geq3$, and
$\min\{\omega_z,\omega_w\}\leq d$ gives
$\kappa(f)\leq4d^4+4d^3-6d^2$.  The coefficient
$4+4/d-6/d^2$ assumes its largest value for an integral $d\geq1$ at $d=3$.
Its value there is $14/3$, and hence
$\kappa(f)\leq\frac{14}{3}d^4$ under this additional condition.  When
$\Delta=d\Delta_2$, direct substitution gives the still sharper formula
$\kappa(f)\leq4d^4-2d^3$.

 
\section{A Newton-Polygon Bound for Multiple Nodes of Amoeba Contours}
 
Let $\Delta\subset\mathbb R^2$ be a two-dimensional lattice polygon.  Let
$f\in\mathbb C[z^{\pm1},w^{\pm1}]$ have Newton polygon $\Delta$, put
$C_f=\{f=0\}\subset(\mathbb C^*)^2$, and let $\overline C_f\subset X_\Delta$
be its toric compactification.  We interpret as before $\mathscr C_\Delta$ as the class
of curves for which $C_f$ is smooth, the logarithmic critical locus is a
smooth real curve, and $\overline C_f$ is torically nondegenerate.  In
particular, $\overline C_f$ is smooth, it avoids the zero-dimensional toric
orbits, and it meets the toric boundary divisors transversely.

\begin{definition}
Let $s\geq2$.  A point $x\in\mathcal K_f$ is an $s$-node if there are exactly
$s$ distinct points $p_1,\ldots,p_s\in S_f$ mapping to $x$, the restriction
$\Log|_{S_f}$ is immersive at every $p_i$, and the $s$ resulting tangent
lines to the contour at $x$ are pairwise distinct.  Thus every two of the
branches meet transversely.  Let $N_s(f)$ be the number of such points.  The
same estimate below applies to the number of points through which at least
$s$ transverse branches pass.
\end{definition}

\begin{theorem}
For every $C_f\in\mathscr C_\Delta$ and every $s\geq2$, the number of isolated
$s$-nodes of its amoeba contour satisfies
$N_s(f)\leq\left\lfloor 8n^2\omega_z\omega_w/(s(s-1))\right\rfloor$.
Therefore,
$N_s(f)\leq\left\lfloor8n^4/(s(s-1))\right\rfloor
=\left\lfloor128\Area(\Delta)^4/(s(s-1))\right\rfloor$.

If $f$ is a polynomial of projective degree $d$, so that after a monomial
translation $\Delta\subset d\Delta_2$, then
$N_s(f)\leq\left\lfloor8d^6/(s(s-1))\right\rfloor$.  These bounds count only
isolated transverse multiple points and remain valid after deleting
positive-dimensional overlap components of the two-lift coincidence scheme.
\end{theorem}

\begin{proof}
The logarithmic Gauss map extends to a morphism
$\gamma_f:\overline C_f\to\mathbb P^1$ given on the torus by
$\gamma_f(z,w)=[zf_z:wf_w]$.  Toric nondegeneracy gives
$\deg\gamma_f=n$.  Let $\overline C_f^\sigma$ be the coefficient-conjugate
curve, with coordinates $(\zeta,\eta)$, and denote its logarithmic Gauss map
by $\gamma_f^\sigma$.  Form the fiber product
$\Gamma=\overline C_f\mathbin{\times}_{\mathbb P^1}\overline C_f^\sigma$.
In torus coordinates it is cut out by $f(z,w)=0$,
$\overline f(\zeta,\eta)=0$, and the equality
$\gamma_f(z,w)=\gamma_f^\sigma(\zeta,\eta)$.  Let
$\nu:D\to\Gamma$ be the normalization.  The curve $D$ is allowed to be
disconnected.

Complex conjugation defines an antiholomorphic involution on the fiber
product.  Its fixed locus is characterized by
$(\zeta,\eta)=(\overline z,\overline w)$, and its projection to the first
factor is precisely the logarithmic critical locus.  On $D$ consider the
meromorphic functions $Z=z\zeta$ and $W=w\eta$.  Along this fixed locus one
has $Z=|z|^2$ and $W=|w|^2$.  It follows that two critical points have the
same logarithmic image if and only if their corresponding real points of $D$
have equal $Z$-values and equal $W$-values.

Let $\delta_Z=\deg Z$ and $\delta_W=\deg W$, where for a disconnected curve
the degree means the sum of the degrees on its connected components.  Each
projection $D\to\overline C_f$ or $D\to\overline C_f^\sigma$ has total degree
$n$.  Since $z$ and $\zeta$ have degree $\omega_z$ on their respective curve
factors, the pullbacks of their pole divisors have degrees at most
$n\omega_z$ each.  The pole divisor of their product has degree at most the
sum, and therefore $\delta_Z\leq2n\omega_z$.  The identical argument for the
second coordinate gives $\delta_W\leq2n\omega_w$.

We now count pairs of distinct points of $D$ having the same image under
$\Phi=(Z,W):D\to\mathbb P^1\times\mathbb P^1$.  On $D\times D$, write the
two points as $q_1$ and $q_2$.  Equality of the first coordinates is the
divisor $Z(q_1)=Z(q_2)$, and equality of the second coordinates is the divisor
$W(q_1)=W(q_2)$.  If $D$ is connected, their divisor classes are respectively
$\delta_Z(F_1+F_2)$ and $\delta_W(F_1+F_2)$, where $F_1$ and $F_2$ are the
two fiber classes.  Since $F_1^2=F_2^2=0$ and $F_1F_2=1$, their intersection
number is $2\delta_Z\delta_W$.

Both equality divisors contain the diagonal $q_1=q_2$.  Removing the diagonal
and every other common curve component does not increase the bidegrees of the
remaining equations.  Bihomogeneous B\'ezout intersection on each product of
connected components of $D$ therefore shows that the isolated ordered
off-diagonal solutions are at most $2\delta_Z\delta_W$.  For completeness,
on one connected component of genus $g_D$, when no other common component is
present, subtracting the diagonal gives the more precise residual intersection
number
$2\delta_Z\delta_W-2\delta_Z-2\delta_W+2-2g_D$, which is no larger than
$2\delta_Z\delta_W$.  If $D$ has several components, applying the same
calculation to every product of two components and summing gives the same
upper bound in terms of the total degrees.
The involution $(q_1,q_2)\mapsto(q_2,q_1)$ acts freely on the off-diagonal
solutions.  Hence the number of isolated unordered pairs of distinct points
with equal $Z$- and $W$-coordinates is at most
$\delta_Z\delta_W\leq4n^2\omega_z\omega_w$ (for more details, see Appendix C).

An $s$-node has $s$ distinct critical lifts.  Every unordered pair among
these lifts is an off-diagonal coincidence, and distinct $s$-nodes supply
disjoint sets of pairs because their logarithmic images are different.  Each
$s$-node therefore consumes exactly $\binom{s}{2}=s(s-1)/2$ unordered pairs.
It follows that
$\binom{s}{2}N_s(f)\leq4n^2\omega_z\omega_w$, which is equivalent to the
first assertion of the theorem.

Every coordinate width of a two-dimensional lattice polygon is at most its
normalized area, so $\omega_z,\omega_w\leq n$.  Substitution gives
$N_s(f)\leq8n^4/(s(s-1))$.  Since $n=2\Area(\Delta)$, this is the same as
$N_s(f)\leq128\Area(\Delta)^4/(s(s-1))$.

If the projective degree is $d$, the containment
$\Delta\subset d\Delta_2$ gives $n\leq d^2$ and
$\omega_z,\omega_w\leq d$.  The polygonal bound then becomes
$N_s(f)\leq8d^6/(s(s-1))$.  Taking integer parts completes the proof.
\end{proof}

\begin{remark}
For $s=2$, the theorem gives $N_2(f)\leq4n^2\omega_z\omega_w$.  For $s=3$ it
gives $N_3(f)\leq(4/3)n^2\omega_z\omega_w$, and for $s=4$ it gives
$N_4(f)\leq(2/3)n^2\omega_z\omega_w$.  The factor $1/\binom{s}{2}$ is forced
by the combinatorics of a transverse intersection of $s$ branches: such a
point simultaneously accounts for one coincidence for each pair of branches.
\end{remark}


 \section{The Exact Bidegree of the Complexified Critical-Value Image}

Let $F(T,X,Y)$ be a homogeneous polynomial of degree $d$ and let
$C=\{F=0\}\subset\mathbb P^2$ be smooth and transverse to the three
coordinate lines.  Its affine torus part is defined by
$f(z,w)=F(1,z,w)$, where $z=X/T$ and $w=Y/T$.  Thus the Newton polygon is the
full triangle $\Delta=d\Delta_2$.  We assume throughout that the toric
compactification is nondegenerate.  Let $C^\sigma$ be the curve defined by
the polynomial obtained by conjugating the coefficients of $F$, and use
$(\zeta,\eta)$ for its affine torus coordinates.

The logarithmic Gauss map is
$\gamma:C\to\mathbb P^1$, with
$\gamma=[XF_X:YF_Y]$.  Its conjugate counterpart is denoted by
$\gamma^\sigma:C^\sigma\to\mathbb P^1$.  The complexified logarithmic
critical curve is the fiber product
$\Gamma=C\mathbin{\times}_{\mathbb P^1}C^\sigma$, and we write
$\nu:D\to\Gamma$ for its normalization.  The curve $D$ may have several
components.  On $D$ consider the rational map
$\Phi=(Z,W):D\to\mathbb P^1\times\mathbb P^1$, where
$Z=z\zeta$ and $W=w\eta$.  On the real fixed locus one has
$Z=|z|^2$ and $W=|w|^2$, so $\Phi$ is the algebraic complexification of the
critical-value parametrization before logarithms are taken.

For a curve in $\mathbb P^1\times\mathbb P^1$, we use coordinate bidegree:
the first entry is the degree of the first projection and the second entry is
the degree of the second projection.  For a reducible image with
multiplicities, the bidegree is understood componentwise and then added.

\begin{theorem}
The map $\gamma$ has degree $d^2$.  The two meromorphic functions on $D$ have
the exact total degrees
$\deg_D Z=2d^3$ and $\deg_D W=2d^3$.  Consequently, the pushforward cycle
$\Phi_*[D]$ has exact coordinate bidegree $(2d^3,2d^3)$.

If $D$ is irreducible and $\Phi$ has generic degree $e$ onto its reduced image
$B=\Phi(D)$, then $B$ has coordinate bidegree
$(2d^3/e,2d^3/e)$.  In particular, if $\Phi$ is birational onto $B$, then the
reduced image has exact bidegree $(2d^3,2d^3)$.
\end{theorem}

\begin{proof}
The two homogeneous expressions $XF_X$ and $YF_Y$ are sections of
$\mathcal O_C(d)$.  They have no common zero on $C$.  Indeed, in the torus a
common zero would give $F=F_X=F_Y=0$, contradicting smoothness.  On the
coordinate boundary, transversality and toric nondegeneracy exclude a common
zero.  The logarithmic Gauss map is therefore a base-point-free pencil in
$\mathcal O_C(d)$.  Since $\deg\mathcal O_C(1)=d$, one obtains
$\deg\gamma=\deg\mathcal O_C(d)=d^2$.

Both projections $\pi_1:D\to C$ and $\pi_2:D\to C^\sigma$ consequently have
total degree $d^2$, with component multiplicities included.  On $C$, the
function $z=X/T$ has a zero divisor of degree $d$ on $X=0$ and a pole divisor
of degree $d$ on $T=0$.  Hence
$\pi_1^*\operatorname{Pole}(z)$ has degree $d^2d=d^3$.  The same calculation
on the second factor gives
$\deg\pi_2^*\operatorname{Pole}(\zeta)=d^3$.

It remains to verify that the sum of these pole divisors is not shortened by
cancellation with a zero from the other factor.  This is where the boundary
values of the logarithmic Gauss map enter.  On $X=0$, one has
$\gamma=[0:1]$.  On $Y=0$, one has $\gamma=[1:0]$.  On $T=0$, Euler's
identity $TF_T+XF_X+YF_Y=dF$ restricts on $C$ to
$XF_X+YF_Y=0$, so $\gamma=[1:-1]$.  These three values are pairwise distinct.

A zero of $z$ belongs to the fiber of $\gamma$ over $[0:1]$, whereas a pole
of $\zeta$ belongs to the fiber of $\gamma^\sigma$ over $[1:-1]$.  Such two
points cannot form a point of the fiber product, because points of the fiber
product have equal Gauss values.  The same reasoning excludes cancellation
between a pole of $z$ and a zero of $\zeta$.  Therefore
$\operatorname{Pole}(Z)=
\pi_1^*\operatorname{Pole}(z)+
\pi_2^*\operatorname{Pole}(\zeta)$, including multiplicities, and its degree
is exactly $2d^3$.  Thus $\deg_DZ=2d^3$.

Replacing $X$ by $Y$ gives the identical conclusion for $W$.  Its zeros lie
over the Gauss value $[1:0]$, while its poles lie over $[1:-1]$, so again
there is no zero--pole cancellation.  Hence $\deg_DW=2d^3$.

By the projection formula, the degree of the first coordinate on the
pushforward cycle $\Phi_*[D]$ equals $\deg_DZ$, and the degree of its second
coordinate equals $\deg_DW$.  This proves that its bidegree is
$(2d^3,2d^3)$.  If $D$ is irreducible and $\Phi$ has generic degree $e$, then
$\Phi_*[D]=e[B]$.  Dividing the two entries by $e$ gives the asserted
bidegree of the reduced image.
\end{proof}

This theorem also explains why an exact bidegree of the reduced image cannot
be specified from membership in $\mathscr C_\Delta$ alone.  Those hypotheses
do not require the fiber product to be irreducible and do not require
$\Phi$ to be birational onto its image.  What is determined without any such
extra condition is the bidegree of the pushforward cycle.  If
$D=\coprod_jD_j$, if $B_j=\Phi(D_j)$, and if $e_j$ is the generic degree of
$D_j\to B_j$, then the exact statement is
$\sum_je_j\deg(Z|_{B_j})=2d^3$ and
$\sum_je_j\deg(W|_{B_j})=2d^3$.

\begin{proposition}
Assume that $F$ has real coefficients.  Then the fiber product
$C\times_{\mathbb P^1}C$ contains the diagonal component $\Delta_C$.  Its
image under $\Phi$ has bidegree $(2d,2d)$.  The complementary off-diagonal
cycle therefore has exact bidegree
$(2d^3-2d,2d^3-2d)$.
\end{proposition}

\begin{proof}
On the diagonal one has $(\zeta,\eta)=(z,w)$ algebraically, so
$\Phi|_{\Delta_C}=(z^2,w^2)$.  Since $z$ and $w$ each have degree $d$ on
$C$, their squares each have degree $2d$.  The diagonal image cycle thus has
bidegree $(2d,2d)$.  Subtraction from the total pushforward bidegree
$(2d^3,2d^3)$ leaves $(2d^3-2d,2d^3-2d)$.  This calculation is a cycle
calculation and remains valid if the residual fiber product is reducible.
\end{proof}

The diagonal is therefore a systematic component for real coefficients, but
it contributes only order $d$.  Removing it changes $2d^3$ into
$2d^3-2d$ and does not alter the leading cubic term.  The three toric boundary
divisors do not produce curve components in $D$, because each of them meets
$C$ in finitely many points.  They also produce no cancellation, as shown by
the three distinct Gauss values $[0:1]$, $[1:0]$, and $[1:-1]$.
There remains one possible mechanism by which the reduced image could have a
smaller bidegree: the map $\Phi$ could have generic degree $e>1$.  This is not
a cancellation in the divisors of $Z$ and $W$; it is a multiple-cover
phenomenon.  No lower bound growing with $d$ for such an $e$ follows from the
definition of $\mathscr C_\Delta$.  In particular, the class does not force
$e$ to be of order $d$, which would be needed to reduce a cubic coordinate
degree to a quadratic one.  Under the natural birationality hypothesis
$e=1$, the full cubic bidegree survives exactly.
The consequence for the node estimate is decisive.  A curve of coordinate
bidegree $(a,b)$ in $\mathbb P^1\times\mathbb P^1$ has arithmetic genus
$(a-1)(b-1)$ when it is reduced and irreducible.  With
$a=b=2d^3$, this scale is $(2d^3-1)^2=4d^6-4d^3+1$.  In the real-coefficient
off-diagonal calculation, the corresponding product is
$(2d^3-2d)^2=4d^6-8d^4+4d^2$.  Both expressions retain a leading term of
order $d^6$.


\section*{Appendix A: Smooth complete-intersection point of $\Gamma_f$'}

Let $f(z,w)$ be a Laurent polynomial and let
$f^\dagger(\zeta,\eta)$ denote its conjugate-coefficient copy, in which
$\zeta$ and $\eta$ are treated as independent complex variables. Put
\(
 P=zf_z,\qquad Q=wf_w,\,
 P^\dagger=\zeta f^\dagger_\zeta,\,
 Q^\dagger=\eta f^\dagger_\eta,
\)
and define
\(
 K=PQ^\dagger-QP^\dagger.
\)
The complexified logarithmic critical curve is
\[
 \Gamma_f=V(f,f^\dagger,K)
 \subset (\C^*)^4,
\]
where the coordinates of the ambient space are $(z,w,\zeta,\eta)$.

The ambient variety $(\C^*)^4$ is a smooth complex manifold of dimension
$4$. Since $\Gamma_f$ is defined by three equations, its expected complex
dimension is $4-3=1$. Thus one expects $\Gamma_f$ to be a complex curve.
Nevertheless, three equations do not automatically cut out a smooth curve:
their differentials may fail to be independent, or the common zero set may
have a component of dimension greater than one. 
Smooth
complete-intersection point of $\Gamma_f$, excludes precisely these
degeneracies.
More explicitly, let
\(
 p=(z_0,w_0,\zeta_0,\eta_0)\in\Gamma_f.
\)
The point $p$ is a smooth complete-intersection point of $\Gamma_f$ if the
three covectors
\(
 df_p,\, df^\dagger_p,\, dK_p
\)
are linearly independent in the cotangent space
$T_p^*((\C^*)^4)$. Equivalently, the Jacobian matrix
\[
 \mathcal J_{\Gamma_f}(p)=
 \begin{pmatrix}
 f_z & f_w & 0 & 0\\
 0 & 0 & f^\dagger_\zeta & f^\dagger_\eta\\
 K_z & K_w & K_\zeta & K_\eta
 \end{pmatrix}_{p}
\]
has rank $3$. By the holomorphic implicit-function theorem, this rank
condition implies that, in a neighbourhood of $p$, the common zero set of
$f$, $f^\dagger$, and $K$ is a smooth complex submanifold of codimension
$3$. Consequently,
\(
 \dim_{\C,p}\Gamma_f=4-3=1.
\)
This is the meaning of ``complete intersection'' in the sentence: locally at
$p$, the three displayed equations impose three independent conditions and
cut out a set of exactly the expected codimension. The smoothness at $p$ means
that this local one-dimensional set has no singularity at $p$.
The tangent space to a common zero set is obtained by differentiating its
defining equations. Hence
\(
 T_p\Gamma_f
 =\ker(df_p)\cap\ker(df^\dagger_p)\cap\ker(dK_p).
\)
Since the three differentials are independent in a four-dimensional
ambient space, this common kernel has complex dimension one. It is therefore
the tangent line to $\Gamma_f$ at $p$. 

If $\operatorname{rank}\mathcal J_{\Gamma_f}(p)<3$, then the common kernel has
dimension at least two and is only the Zariski tangent space defined by the
linearized equations. It need not be the tangent line of a smooth curve. Such
a point may be singular, may lie on several local branches, or may belong to
an excess-dimensional component. Therefore, the tangent-line argument used
later in the proof is justified only after restricting to the smooth
complete-intersection locus.
To see how this condition is used, consider the algebraic critical-value map
\(
 \mu:\Gamma_f\longrightarrow\C^2,
 \,
 \mu(z,w,\zeta,\eta)=(R,S)=(z\zeta,w\eta).
\)
On the real form $\zeta=\overline z$ and
$\eta=\overline w$, one has $R=|z|^2$ and $S=|w|^2$. Composing with
\[
 (R,S)\longmapsto
 \left(\frac12\log R,\frac12\log S\right)
\]
recovers the logarithmic contour map. Since the latter change of coordinates
is a real-analytic diffeomorphism on $\mathbb R_{>0}^2$, it does not change
the local ramification condition.

At a smooth complete-intersection point $p$, the restriction
$d\mu_p|_{T_p\Gamma_f}$ vanishes if and only if both $dR_p$ and $dS_p$
vanish on the one-dimensional tangent space. Since
\(
 T_p\Gamma_f
 =\ker(df_p)\cap\ker(df^\dagger_p)\cap\ker(dK_p),
\)
the vanishing of $dR_p$ on this tangent line is equivalent to the linear
dependence of
$df_p,df^\dagger_p,dK_p,dR_p$. It is therefore expressed by
\[
 J_R(p)=
 \det\!\begin{pmatrix}
 f_z&f_w&0&0\\
 0&0&f^\dagger_\zeta&f^\dagger_\eta\\
 K_z&K_w&K_\zeta&K_\eta\\
 R_z&R_w&R_\zeta&R_\eta
 \end{pmatrix}_{p}=0.
\]
Similarly, the vanishing of $dS_p$ on the tangent line is equivalent to
\[
 J_S(p)=
 \det\!\begin{pmatrix}
 f_z&f_w&0&0\\
 0&0&f^\dagger_\zeta&f^\dagger_\eta\\
 K_z&K_w&K_\zeta&K_\eta\\
 S_z&S_w&S_\zeta&S_\eta
 \end{pmatrix}_{p}=0.
\]
Consequently, on the smooth complete-intersection locus of $\Gamma_f$, the
ramification points of $\mu$ are exactly the common zeros of $J_R$ and $J_S$.
This determinant characterization would not by itself be sufficient at a
point where the first three rows have rank smaller than three, because there
the Zariski tangent space is larger than a line and the determinants may
vanish merely because $\Gamma_f$ is singular.


\section{Appendix B: The diagonal saturation and the ordered off-diagonal double-point scheme}

Let $f(z,w)$ be a Laurent polynomial and let
$f^\dagger(\zeta,\eta)$ be its conjugate-coefficient copy.  Put
$P=zf_z$, $Q=wf_w$,
$P^\dagger=\zeta f^\dagger_\zeta$, and
$Q^\dagger=\eta f^\dagger_\eta$.  The complexified logarithmic critical
curve is
$
\Gamma_f=V(f,f^\dagger,PQ^\dagger-QP^\dagger)\subset(\mathbb C^*)^4.
$
On this curve consider
$\mu(z,w,\zeta,\eta)=(z\zeta,w\eta)$.  On the real form
$(\zeta,\eta)=(\overline z,\overline w)$, this becomes
$(|z|^2,|w|^2)$, so two critical points have the same logarithmic value
exactly when they have the same $\mu$-value.
Take two independent points $p_1,p_2$ of $\Gamma_f$.  In the doubled Laurent
ring, the raw ideal is generated by the two copies of the equations defining
$\Gamma_f$ together with
$
z_1\zeta_1-z_2\zeta_2=0,
\, 
w_1\eta_1-w_2\eta_2=0.
$
Its zero scheme is the fibre product
$\Gamma_f\times_{(\mathbb C^*)^2}\Gamma_f$.  It contains all pairs of
distinct critical points with the same critical value, but it also contains
the entire diagonal $p_1=p_2$.
The diagonal ideal is
$
I_{\mathrm{diag}}=(z_1-z_2,w_1-w_2,
\zeta_1-\zeta_2,\eta_1-\eta_2).
$
The ordered off-diagonal ideal is defined by
$
I_{\mathrm{off}}
=I_{\mathrm{raw}}:I_{\mathrm{diag}}^\infty.
$
The geometric meaning of saturation is
$V(I:J^\infty)=\overline{V(I)\setminus V(J)}$.  Hence
$V(I_{\mathrm{off}})$ is the scheme-theoretic closure of the locus of pairs
$(p_1,p_2)$ satisfying $\mu(p_1)=\mu(p_2)$ and $p_1\neq p_2$.  Saturation
removes components supported on the diagonal, including embedded components,
but the closure of the off-diagonal locus may still meet the diagonal at
limiting ramification points.
If the equations are homogenized in the Cox ring of a toric compactification,
two further operations must be distinguished.  Saturation by the irrelevant
ideal $B_\Sigma$ makes the homogeneous ideal define the correct toric
subscheme, whereas saturation by the boundary monomial
$x_\partial=\prod_{\rho\in\Sigma(1)}x_\rho$ removes components supported
entirely on the toric boundary.  A compactified version is therefore
$
I_{\mathrm{off}}^{\mathrm{tor}}
=\left(\left(
I_{\mathrm{raw}}^h:(I_{\mathrm{diag}}^h)^\infty
\right):B_\Sigma^\infty\right):x_\partial^\infty.
$
Working directly in the Laurent ring avoids the irrelevant Cox locus and
restricts the computation to the dense torus.
The involution $\tau(p_1,p_2)=(p_2,p_1)$ preserves the off-diagonal scheme.
If every coincident value has exactly two distinct reduced preimages, then it
contributes the two ordered pairs $(p_1,p_2)$ and $(p_2,p_1)$; the number of
unordered double values is therefore half the length of the ordered scheme.
This division requires reducedness, a free involution, and the absence of
triple or higher fibres.  A fibre with three distinct preimages contributes
six ordered pairs, and nontransverse intersections must be counted with their
scheme multiplicities.
Thus $I_{\mathrm{off}}$ is the appropriate algebraic object for recording
ordered distinct critical preimages with the same logarithmic critical value.


\section{Appendix C: Unordered-Pair Bound}

Let $D$ be the normalization of the complexified logarithmic critical curve,
and let
$$
\Phi=(Z,W):D\longrightarrow\mathbb P^1\times\mathbb P^1
$$
be the critical-value map, where $Z=z\zeta$ and $W=w\eta$.  The curve $D$ is
allowed to be disconnected.  Write $\delta_Z=\deg_D Z$ and
$\delta_W=\deg_D W$, where the degree on a disconnected curve is the sum of
the degrees on its connected components.
An off-diagonal coincidence is a pair $(q_1,q_2)\in D\times D$ satisfying
$q_1\neq q_2$, $Z(q_1)=Z(q_2)$, and $W(q_1)=W(q_2)$.  We count unordered
pairs, so $(q_1,q_2)$ and $(q_2,q_1)$ represent the same coincidence.  The
claim to be proved is that the number of isolated unordered coincidences is at
most $\delta_Z\delta_W$.  Combining this with
$\delta_Z\leq2n\omega_z$ and $\delta_W\leq2n\omega_w$ gives
$\delta_Z\delta_W\leq4n^2\omega_z\omega_w$.

\begin{proposition}
Let $D$ be a smooth compact curve, possibly disconnected, and let
$\Phi=(Z,W):D\to\mathbb P^1\times\mathbb P^1$ be a morphism.  After removing
the diagonal and every positive-dimensional component of
$D\times_{\Phi(D)}D$, let $M$ be the number of isolated unordered pairs of
distinct points having the same image.  Then
$M\leq\delta_Z\delta_W$.
\end{proposition}

\begin{proof}
Let $B=\Phi(D)_{\mathrm{red}}$ be the reduced image curve.  Components of $D$
that map with degree larger than one onto an image component produce
positive-dimensional off-diagonal correspondences.  After those
correspondences have been removed, every isolated coincidence is represented
on the normalization of $B$.  Thus it is enough to work with the normalization
map $\nu:\widetilde B\to B$, where $\widetilde B$ is the disjoint union of the
normalizations of the irreducible components of $B$.
Let $x\in B$ be a singular point and suppose that its inverse image under
$\nu$ consists of $r_x$ distinct points.  These points represent the
$r_x$ analytic branches of $B$ through $x$.  They produce
$\binom{r_x}{2}$ unordered pairs of distinct normalization points having the
same image.
Let $\delta_x(B)$ be the local delta invariant.  Algebraically it is
$\delta_x(B)=\dim_{\mathbb C}
(\nu_*\mathcal O_{\widetilde B,x}/\mathcal O_{B,x})$.  If the branches at
$x$ are $B_1,\ldots,B_{r_x}$, the branch formula for the delta invariant is
$$
\delta_x(B)=\sum_{i=1}^{r_x}\delta_x(B_i)
+\sum_{1\leq i<j\leq r_x}I_x(B_i,B_j),
$$
where $I_x(B_i,B_j)$ is the local intersection multiplicity of two distinct
branches.  Every branch delta invariant is nonnegative, and every
$I_x(B_i,B_j)$ is at least one.  Therefore
$\delta_x(B)\geq\binom{r_x}{2}$.  Equality holds when all branches are smooth
and every pair meets transversely, as at an ordinary multiple point.
Summing over the singular points of $B$ shows that the total number of
isolated unordered pairs is bounded by the total delta invariant:
$M\leq\sum_{x\in\operatorname{Sing}(B)}\delta_x(B)$.

Let $B_1,\ldots,B_c$ be the irreducible components of $B$, and let $g_i$ be
the genus of the normalization of $B_i$.  The normalization exact sequence
gives
$\di
\sum_{x\in\operatorname{Sing}(B)}\delta_x(B)
=p_a(B)-\sum_{i=1}^c g_i+c-1.
$
Let $(a,b)$ be the coordinate bidegree of $B$, meaning that $a$ is the total
degree of the first projection and $b$ is the total degree of the second
projection.  Adjunction on $\mathbb P^1\times\mathbb P^1$ gives
$p_a(B)=(a-1)(b-1)$.  Since every $g_i$ is nonnegative,
$\di
\sum_x\delta_x(B)\leq(a-1)(b-1)+c-1.
$
Every irreducible curve component of $B$ has positive degree under at least
one of the two coordinate projections.  If its coordinate bidegree is
$(a_i,b_i)$, then $a_i+b_i\geq1$.  Summing over the components gives
$c\leq\sum_i(a_i+b_i)=a+b$.  It follows that
$\di
\sum_x\delta_x(B)
\leq(a-1)(b-1)+a+b-1=ab.
$
Therefore, $M\leq ab$.
It remains to compare the bidegree of the reduced image with the degrees on
$D$.  The degree $a$ of the first projection of $B$ is no larger than the
degree of $Z$ on $D$, and similarly $b\leq\delta_W$.  Indeed, on each source
component the degree of a coordinate function equals the generic covering
degree onto the image component multiplied by the corresponding coordinate
degree of that image component.  Summation gives $a\leq\delta_Z$ and
$b\leq\delta_W$.  Hence
$M\leq ab\leq\delta_Z\delta_W$, as required.
\end{proof}
We now specialize this statement to the complexified logarithmic critical
curve.  Let $C_f$ be the torically compactified curve and let
$\gamma_f:C_f\to\mathbb P^1$ be its logarithmic Gauss map.  Put
$n=\deg\gamma_f=2\operatorname{Area}(\Delta)$.  The two projections from $D$
to $C_f$ and to its coefficient-conjugate curve have total degree $n$.
The meromorphic function $z$ has degree $\omega_z$ on $C_f$, and $\zeta$ has
the same degree on the conjugate curve.  The pole divisor of
$Z=z\zeta$ is bounded by the sum of the pullbacks of the pole divisors of
$z$ and $\zeta$.  Each pullback has degree $n\omega_z$.  Thus
$\delta_Z\leq2n\omega_z$.  Boundary zero--pole cancellation can make this
inequality strict, but it cannot increase the degree.  The identical argument
gives $\delta_W\leq2n\omega_w$.
Combining the proposition with these two estimates gives
$
M\leq\delta_Z\delta_W
\leq(2n\omega_z)(2n\omega_w)
=4n^2\omega_z\omega_w.
$
This is the claimed inequality.

\begin{remark}
At a transverse $s$-fold contour point, the $s$ distinct critical lifts give
exactly $\binom{s}{2}$ unordered pairs.  Since different contour values give
disjoint collections of pairs, the same proof yields
$N_s(f)\binom{s}{2}\leq4n^2\omega_z\omega_w$, and hence
$N_s(f)\leq8n^2\omega_z\omega_w/(s(s-1))$.
\end{remark}

 \end{document}